\documentclass[psamsfonts,reqno]{amsart}
\usepackage{amssymb,eucal,latexsym,graphics,xy}
\xyoption{all}

\theoremstyle{plain}

\newtheorem{lemma}{Lemma}[section]
\newtheorem{theorem}[lemma]{Theorem}
\newtheorem{proposition}[lemma]{Proposition}
\newtheorem{corollary}[lemma]{Corollary}

\newtheorem*{sclaim}{Claim}

\newtheorem*{stat}{\name}
\newcommand{\name}{testing}

\theoremstyle{definition}
\newtheorem{definition}[lemma]{Definition}
\newtheorem{example}[lemma]{Example}
\newtheorem{problem}{Problem}

\theoremstyle{remark}
\newtheorem{remark}[lemma]{Remark}

\newtheorem*{note}{Note}

\newenvironment{all}[1]{\renewcommand{\name}{#1}\begin{stat}}
                        {\end{stat}}

\newcommand{\qedc}{{\qed}~{\rm Claim~{\theclaim}.}}
\newcommand{\qedsc}{{\qed}~{\rm Claim.}}

\newenvironment{scproof}
{\begin{proof}[Proof of Claim.]}
{\qedsc\renewcommand{\qed}{}\end{proof}}

\numberwithin{equation}{section}
\numberwithin{figure}{section}

\newcommand{\eps}{\varepsilon}
\newcommand{\ol}[1]{\overline{#1}}
\newcommand{\les}{\leqslant}

\newcommand{\id}{\mathrm{id}}

\newcommand{\tvi}{\vrule height 10pt depth 4pt width 0pt}

\newcommand{\pup}[1]{\textup{(}{#1}\textup{)}}

\newcommand{\Mat}{\mathrm{M}}
\newcommand{\Idp}{\mathrm{I}}

\DeclareMathOperator{\card}{card}
\DeclareMathOperator{\kur}{kur}
\DeclareMathOperator{\tr}{tr}

\DeclareMathOperator{\Dim}{Dim}
\DeclareMathOperator{\forg}{\operatorname{\Pi}}
\DeclareMathOperator{\lift}{\operatorname{\Gamma}}

\newcommand{\xF}{\operatorname{\mathbf{F}}}
\newcommand{\NN}{\mathbb{N}}
\newcommand{\ZZ}{\mathbb{Z}}
\newcommand{\QQ}{\mathbb{Q}}
\newcommand{\RR}{\mathbb{R}}
\newcommand{\CC}{\mathbb{C}}
\renewcommand{\AA}{\mathbb{A}}
\newcommand{\DD}{\mathbb{D}}
\newcommand{\II}{\mathbb{I}}
\newcommand{\JJ}{\mathbb{J}}
\newcommand{\into}{\hookrightarrow}
\newcommand{\onto}{\twoheadrightarrow}
\newcommand{\simst}{\overset{*}{\sim}}

\newcommand{\ba}{\boldsymbol{a}}
\newcommand{\bb}{\boldsymbol{b}}
\newcommand{\bc}{\boldsymbol{c}}

\newcommand{\be}{\boldsymbol{e}}
\renewcommand{\bf}{\boldsymbol{f}}
\newcommand{\bg}{\boldsymbol{g}}
\newcommand{\bh}{\boldsymbol{h}}
\newcommand{\bs}{\boldsymbol{s}}

\newcommand{\bA}{\boldsymbol{A}}
\newcommand{\bB}{\boldsymbol{B}}

\newcommand{\bX}{\boldsymbol{X}}

\newcommand{\xt}{\mathbf{t}}

\newcommand{\jirr}{join-ir\-re\-duc\-i\-ble}

\newcommand{\Chom}{C*-ho\-mo\-mor\-phism}
\newcommand{\Vhom}{V-ho\-mo\-mor\-phism}
\newcommand{\Cemb}{C*-em\-bed\-ding}
\newcommand{\cm}{commutative monoid}

\newcommand{\res}{\mathbin{\restriction}}
\newcommand{\norm}[1]{\Vert{#1}\Vert}
\newcommand{\Ktzero}{K(\xt)_{(0)}}

\newcommand{\gQa}{\operatorname{\Theta_{\mathrm{a}}}}
\newcommand{\gQr} {\operatorname{\Theta_{\mathrm{r}}}}
\newcommand{\gYa}{\Psi_{\mathrm{a}}}
\newcommand{\gYr}{\Psi_{\mathrm{r}}}

\newcommand{\set}[1]{\left\{#1\right\}}
\newcommand{\setm}[2]{\set{{#1}\mid{#2}}}
\newcommand{\famm}[2]{\left({#1}\mid{#2}\right)}

\DeclareMathOperator{\rC}{C}
\DeclareMathOperator{\rF}{F}
\DeclareMathOperator{\rJ}{J}
\DeclareMathOperator{\rK}{K}
\DeclareMathOperator{\rV}{V}
\newcommand{\rVu}{\rV^{1}}

\newcommand{\cA}{\mathcal{A}}

\newcommand{\cF}{\mathcal{F}}
\newcommand{\vL}{\vec{L}}
\newcommand{\vLu}{{\vec{L}}^{\mathrm{u}}}
\newcommand{\cR}{\mathcal{R}}

\newcommand{\sM}{\mathsf{M}}
\newcommand{\vR}{{\vec{R}}}
\newcommand{\vA}{{\vec{A}}}

\newcommand{\vD}{{\vec{D}}}
\newcommand{\vDu}{{\vec{D}}^{\mathrm{u}}}
\newcommand{\vE}{{\vec{E}}}
\newcommand{\vK}{{\vec{K}}}

\newcommand{\sep}{\mathrm{sep}}
\newcommand{\CM}{\mathbf{CMon}}
\newcommand{\Ring}{\mathbf{Ring}}
\newcommand{\wVSem}{\mathbf{wVSem}}
\newcommand{\Exch}{\mathbf{Exch}}
\newcommand{\Reg}{\mathbf{Reg}}

\newcommand{\meas}{\mathrm{meas}}
\newcommand{\pmeas}{\mathrm{pmeas}}
\newcommand{\Vmeas}{\mathrm{Vmeas}}
\newcommand{\Vpmeas}{\mathrm{Vpmeas}}

\newcommand{\AF}{\mathbf{AF}}
\newcommand{\xA}{\mathbf{A}}
\newcommand{\xB}{\mathbf{B}}
\newcommand{\xC}{\mathbf{C}}

\newcommand{\xI}{\mathbf{I}}
\newcommand{\xM}{\mathbf{M}}
\newcommand{\xR}{\mathbf{R}}
\newcommand{\xS}{\mathbf{S}}
\newcommand{\xSe}{\mathbf{S}^{\mathrm{emb}}}
\newcommand{\xV}{\mathbf{V}}
\newcommand{\xRR}{\mathbf{RR_0}}
\newcommand{\xRRs}{\xRR^{\sep}}

\newcommand{\rloopd}[2]{\ar@'{@+{[0,0]+(5,-#2)}@+{[0,0]+(#1,0)}@+{{[0,0]+(5,#2)}}}}
\newcommand{\lloopd}[2]{\ar@'{@+{[0,0]+(-5,#2)}@+{[0,0]+(-#1,0)}@+{{[0,0]+(-5,-#2)}}}}

\begin{document}

\author[F. Wehrung]{Friedrich Wehrung}
\address{LMNO, CNRS UMR 6139\\
D\'epartement de Math\'ematiques\\
Universit\'e de Caen\\
14032 Caen Cedex\\
France} \email{wehrung@math.unicaen.fr}
\urladdr{http://www.math.unicaen.fr/\~{}wehrung}

\title[Lifting defects in nonstable K${}_0$-theory]{Lifting defects for nonstable K${}_0$-theory of exchange rings and C*-algebras}

\date{\today}

\begin{abstract}
The assignment (\emph{nonstable $\rK_0$-theory}), that to a ring~$R$ associates the monoid~$\rV(R)$ of Murray-von Neumann equivalence classes of idempotent infinite matrices with only finitely nonzero entries over~$R$, extends naturally to a functor. We prove the following lifting properties of that functor:

\begin{enumerate}
\item There is no functor $\lift$, from simplicial monoids with order-unit with normalized positive \emph{homomorphisms} to exchange rings, such that $\rV\circ\lift\cong\id$.

\item There is no functor $\lift$, from simplicial monoids with order-unit with normalized positive \emph{embeddings} to C*-algebras of real rank~$0$ (resp., von Neumann regular rings), such that $\rV\circ\lift\cong\id$.

\item There is a $\set{0,1}^3$-indexed commutative diagram~$\vD$ of simplicial monoids that can be lifted, with respect to the functor~$\rV$, by exchange rings and by C*-algebras of real rank~$1$, but not by semiprimitive exchange rings, thus neither by regular rings nor by C*-algebras of real rank~$0$.
\end{enumerate}

By using categorical tools (\emph{larders}, \emph{lifters}, \emph{CLL}) from a recent book from the author with P. Gillibert, we deduce that there exists a unital exchange ring of cardinality~$\aleph_3$ (resp., an $\aleph_3$-separable unital C*-algebra of real rank~$1$)~$R$, with stable rank~$1$ and index of nilpotence~$2$, such that~$\rV(R)$ is the positive cone of a dimension group but it is not isomorphic to~$\rV(B)$ for any ring~$B$ which is either a C*-algebra of real rank~$0$ or a regular ring.
\end{abstract}

\maketitle
\section{Introduction}\label{S:Intro}

While attending the August 2010 ``New Trends in Noncommutative Algebra'' conference in Seattle, the author of the present paper was told about the question whether the~$\rK_0$ functor from~AF C*-algebras to dimension groups \emph{splits}, that is, has a (categorical) right inverse. We describe here three reasons why this is not the case. The first and the third reason are more conveniently expressed \emph{via} the \emph{nonstable $\rK_0$-theory} (often also called ``nonstable K-theory'') of a ring~$R$, mainly described by a \cm, usually denoted by~$\rV(R)$. If~$R$ is unital, then~$\rV(R)$ is the monoid of all isomorphism types of finitely generated projective right $R$-modules and~$\rK_0(R)$ is the Grothendieck group of~$\rV(R)$.

\subsection*{Our first reason,} Corollary~\ref{C:FirstNonLift}, implies that the more general question, whether the functor~$\rV$ has a right inverse from simplicial monoids (cf. Section~\ref{S:BasicMon}) to \emph{exchange rings} (cf. Warfield~\cite{Warf72}, and Ara~\cite{Ara97} for the non-unital case), has a negative answer. Actually, Corollary~\ref{C:FirstNonLift} is stated for an even more general class of rings that we call \emph{weakly V-semiprimitive rings}. It also suggests that the original question might get more interesting if one required the morphisms between dimension groups be \emph{one-to-one}. \emph{\textbf{Our second and third reason will address the modified $\mathbf{K_0}$-splitting question}}.

\subsection*{Our second reason,} Proposition~\ref{P:NonSplit}, involves the author's construction, introduced in Wehrung~\cite{NonMeas}, of a dimension group with order-unit, of cardinality~$\aleph_2$, which is not isomorphic to~$\rV(R)$ for any (von~Neumann) regular ring~$R$. If~$\rK_0$ had a right inverse from simplicial groups (cf. Section~\ref{S:BasicMon}) to AF C*-algebras, then it would also have a right inverse from  simplicial groups to regular rings. However, due to the abovementioned counterexample, this is impossible. There are some complications, due to the requirement that the maps between simplicial groups be one-to-one, that we solve by using tools from lattice theory and universal algebra.

\subsection*{Our third reason,} Theorem~\ref{T:NoLift}, is the combinatorial core of our second reason. It involves a commutative diagram~$\vD$ of simplicial monoids, indexed by a cube (i.e., the powerset lattice $\set{0,1}^3$ of a three-element set), that has no lifting, with respect to the functor~$\rV$, by \emph{semiprimitive exchange rings}. In particular, it can be lifted neither by C*-algebras of real rank~$0$ nor by regular rings. In particular, the negative result of Proposition~\ref{P:NonSplit} extends both to regular rings (this could have been deduced directly from~\cite{NonMeas}) and to C*-algebras of real rank~$0$ (this is new).

\subsection*{Lifting the diagrams~$\vD$ and~$\vL$}
Besides those three reasons, we also provide two positive representation results of the abovementioned diagram~$\vD$: while Theorem~\ref{T:NoLift} implies that~$\vD$ cannot be lifted, with respect to the functor~$\rV$, by any commutative diagram of semiprimitive exchange rings, Theorem~\ref{T:LuLift} implies that a certain ``collapsed'' image, denoted by~$\vL$, of the diagram~$\vD$, can be lifted with \emph{exchange rings}, while Theorem~\ref{T:LuLiftC*} implies that~$\vL$ can be lifted with \emph{C*-algebras} (of real rank and stable rank both equal to~$1$). In order to be able to state that~$\vL$ is really a ``diagram'', we need the extended notion of diagram given in Definition~\ref{D:Diagram}.
By ``unfolding''~$\vL$ back to~$\vD$, we obtain (Proposition~\ref{P:NoLiftLu}) that~$\vD$ can also be lifted with exchange rings on the one hand, and with C*-algebras on the other hand. There is no ``best of two worlds'', because a C*-algebra is an exchange ring if{f} it has real rank~$0$, and~$\vD$ cannot be lifted by C*-algebras of real rank~$0$.

\subsection*{The journey back to the transfinite: separating the nonstable $\rK_0$-theory of exchange rings and the one of C*-algebras from the ones of regular rings and C*-algebras of real rank~$\mathbf{0}$}
The reason why we need a diagram indexed by a finite lattice lies in a quite technical result of categorical algebra by Gillibert and the author~\cite{GiWe2} called the \emph{Condensate Lifting Lemma}, CLL in short. CLL makes it possible to turn \emph{diagram counterexamples} to \emph{object counterexamples}, with possible jumps of cardinality. It works only on diagrams indexed by so-called \emph{almost join-semilattices}, in particular for finite lattices.

By applying various instances of CLL to the negative lifting versus positive lifting results of the cube~$\vD$, we obtain (cf. Theorems~\ref{T:CXAKRR0} and~\ref{T:C*CXAKRR0}) that \emph{there exists a unital exchange ring \pup{resp., a unital C*-algebra of real rank~$1$}~$R$, with stable rank~$1$ and index of nilpotence~$2$, such that~$\rV(R)$ is the positive cone of a dimension group and~$\rV(R)$ is not isomorphic to~$\rV(B)$ for any ring~$B$ which is either a C*-algebra of real rank~$0$ or a regular ring}. Due to the cardinality-jumping properties of the ``condensate'' construction underlying CLL, the exchange ring that we construct has~$\aleph_3$ elements, while the C*-algebra that we construct is $\aleph_3$-separable (it has also real rank~$1$). Both the ring and the C*-algebra are stably finite, have stable rank~$1$ and index of nilpotence~$2$, while the nonstable $\rK_0$-theory is the positive cone of a dimension group with order-unit of index~$2$.

Although the present paper is written in the language of ring theory, a substantial part of our results and proofs draw their basic inspiration from related results in universal algebra and lattice theory. For example,
\begin{itemize}
\item Corollary~\ref{C:FirstNonLift} (the ``first reason'') originates in T\r{u}ma and Wehrung \cite[Section~8]{TuWe01}, which is a non-lifting result by universal algebras with respect to the congruence lattice functor.

\item Proposition~\ref{P:NonSplit} (the ``second reason'') originates in Wehrung~\cite{NonMeas}, which solves, in the negative, representation problems with respect to various functors, including~$\rK_0$, by using the lattice structure of the set of all principal right ideals in a regular ring.

\item The cube $\vD$ (cf. Figure~\ref{Fig:DiagrD}), which is the key ingredient of Theorem~\ref{T:NoLift} (the ``third reason''), originates in a cube of Boolean semilattices introduced in T\r{u}ma and Wehrung \cite[Section~3]{TuWe01}.
\end{itemize}

\subsection*{Acknowledgment}
The author is indebted to Pere Ara, Francesc Perera, Enrique Pardo, George Elliott, Chris Phillips, and Ken Goodearl for stimulating interaction and helpful comments at various stages of this paper's preparation. Special thanks are due to Ken Goodearl for his improvement of the author's original argument, which led to Theorem~\ref{T:FirstNonLift}, and to Chris Phillips for most inspiring conversations about C*-algebras and noncommutative rings at the Seattle conference in August~2010.

\section{Organization of the paper}\label{S:Org}
All our rings will be associative, but not necessarily unital. We denote by~$1_R$, or~$1$ if~$R$ is understood, the unit of a ring~$R$.

Although we shall use standard commutative diagrams such as the one, denoted by~$\vD$, introduced in Figure~\ref{Fig:DiagrD}, we shall also use sometimes diagrams of a more general kind, a formal definition of which follows.

\begin{definition}\label{D:Diagram}
A \emph{quiver} is a quadruple $\xI=(V,E,s,t)$ where~$V$ and~$E$ (the ``vertices'' and the ``edges'', respectively) are sets and $s,t\colon E\to V$ (the ``source'' and ``target'' map, respectively). An element $e\in E$ is thus viewed as an ``arrow'' from~$s(e)$ to~$t(e)$.

A \emph{diagram} in a category~$\xC$, \emph{indexed by~$\xI$}, is a functor from~$\xI$ to~$\xC$. Formally, it is given by a pair $(\Phi_{\mathrm{o}},\Phi_{\mathrm{m}})$, where~$\Phi_{\mathrm{o}}$ (resp., $\Phi_{\mathrm{m}}$) is a map from~$V$ (resp., $E$) to the object class (resp., morphism class) of~$\xC$, such that $\Phi_{\mathrm{m}}(e)$ is a morphism from~$\Phi_{\mathrm{o}}(s(e))$ to~$\Phi_{\mathrm{o}}(t(e))$ for any $e\in E$. There is no composition of arrows in~$\xI$, thus nothing to worry about at that level.

The usual definitions of a \emph{natural transformation} and a \emph{natural equivalence} carry over to diagrams without modification. For diagrams $\Phi,\Psi\colon\xI\to\xC$, we denote by $\Phi\cong\Psi$ the natural equivalence of~$\Phi$ and~$\Psi$. For categories~$\xA$ and~$\xB$, we say that a diagram $\lift\colon\xI\to\xA$ \emph{lifts} a diagram $\Psi\colon\xI\to\xB$ with respect to a functor $\Phi\colon\xA\to\xB$ if $\Psi\cong\Phi\lift$ (cf. Figure~\ref{Fig:liftdiag}). In all our lifting problems, we will need to specify additional composition relations among the arrows in the range of~$\Gamma$. (The first such instance in this paper is Theorem~\ref{T:FirstNonLift}: the additional requirement is $h\circ s=h$.)
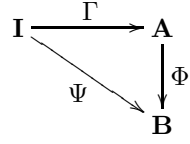
\begin{figure}[htb]
 \[
\xymatrixrowsep{2pc}\xymatrixcolsep{1.5pc}
\def\labelstyle{\displaystyle}
\xymatrix{
\xI\ar[rr]^{\lift}\ar[rrd]_{\Psi} && \xA\ar[d]^{\Phi}\\
&& \xB
}
 \]
\caption{The diagram $\lift$ lifts the diagram $\Psi$ with respect to $\Phi$}
\label{Fig:liftdiag}
\end{figure}
\end{definition}

In many of our applications (but not all of them), $\xA$ will be a category of rings, $\xB$ will be a category of \cm s, and~$\Phi$ will be the nonstable $\rK_0$-theory functor~$\rV$.

We shall now give a brief section by section overview of the paper.

\textbf{In Section~\ref{S:BasicMon}}, we recall the basic monoid-theoretical concepts used in the paper.

\textbf{In Section~\ref{S:rV(R)}}, we recall some basic facts about the functor~$\rV$, from rings to \cm s, describing nonstable $\rK_0$-theory, for either rings or C*-algebras. We also recall basic facts about \emph{semiprimitive} rings.

\textbf{In Section~\ref{S:BasicExch}}, we recall some basic facts about \emph{exchange rings}, including the non-unital case.

\textbf{In Section~\ref{S:FirstDiagr}}, we introduce a notion of orthogonality of elements in a \cm, then the notion of a \emph{weakly V-semiprimitive ring} (\emph{V-semiprimitive rings} will be introduced in Section~\ref{S:Vsemi}). We prove that every exchange ring is weakly V-semiprimitive. We also introduce a diagram (in the sense of Definition~\ref{D:Diagram})~$\vK$ of simplicial monoids with order-unit that has no lifting, with respect to the functor~$\rV$, by weakly V-semiprimitive rings (Theorem~\ref{T:FirstNonLift}). It follows (Corollary~\ref{C:FirstNonLift}) that the functor~$\rV$ has no right inverse from simplicial monoids with order-unit, with normalized monoid \emph{homomorphisms}, to weakly V-semiprimitive rings.

\textbf{In Section~\ref{S:NoSplit}}, we prove (cf. Proposition~\ref{P:NonSplit}) that the functor~$\rV$ has no right inverse from simplicial monoids with unit, and their \emph{embeddings}, to AF C*-algebras. The proof uses universal algebra, lattice theory, and a counterexample of cardinality~$\aleph_2$ constructed in Wehrung~\cite{NonMeas}. As Proposition~\ref{P:NonSplit} will be later superseded by stronger results, we give only an outline of the proof. Nevertheless we chose to keep it there, because a large part of the material in further sections originates from the proof of that result.

\textbf{In Section~\ref{S:MeasRing}} we introduce a tool that will prove essential for our work, the notion of a \emph{\pup{pre}measured ring}. By using some results of Section~\ref{S:rV(R)} and~\ref{S:BasicExch}, we prove (cf. Lemma~\ref{L:ReflSubcat}) that the category of \emph{measured} exchange rings is a reflective subcategory of the one of \emph{premeasured} exchange rings. We also prove (cf. Lemma~\ref{L:ReflSubRR0}) an analogous result for C*-algebras of real rank~$0$.

\textbf{In Section~\ref{S:Vsemi}}, we introduce a common generalization of semiprimitive exchange rings, principal ideal domains, and polynomial rings over fields, that we call \emph{V-semiprimitive rings}. We also observe there that not every exchange ring, and not every C*-algebra, is V-semiprimitive. The C*-algebra~$\DD$ constructed in Example~\ref{Ex:semiprnotV} will play a crucial role in further sections.

\textbf{In Section~\ref{S:Unlift}}, we introduce the main combinatorial object of the paper, namely the diagram~$\vD$. It is a $\set{0,1}^3$-indexed commutative diagram of simplicial monoids, and it originates in the proof of Proposition~\ref{P:NonSplit}. Theorem~\ref{T:NoLift} implies that~$\vD$ has no lifting, with respect to the functor~$\rV$, by any commutative diagram \emph{of V-semiprimitive} rings. This is the central non-representability result of the paper: all the other negative lifting results of the paper, with the exception of Theorem~\ref{T:FirstNonLift}, follow from that one. The proof of Theorem~\ref{T:NoLift} is direct: it does not involve any universal algebra, lattice theory, or transfinite cardinals. We also introduce an auxiliary diagram, denoted there by~$\vE$, of simplicial monoids and monoid \emph{embeddings}, such that any lifting of~$\vE$, with respect to the functor~$\forg$ introduced in Definition~\ref{D:TwoFunct}, gives rise to a lifting of~$\vD$. This implies, in particular, that the functor~$\rV$ has no right inverse from simplicial monoids with order-unit, and their \emph{embeddings}, to C*-algebras of real rank~$0$ (resp., regular rings), see Corollary~\ref{C:NoLift3}.

\textbf{In Section~\ref{S:Collapse}}, we take advantage of the many symmetries underlying the diagram~$\vD$ to collapse it to another diagram (now in the sense of Definition~\ref{D:Diagram}), denoted by~$\vL$, much simpler looking than~$\vD$, such as every lifting of~$\vL$ gives rise to a lifting of~$\vD$ (this is shown in the argument of the proof of Proposition~\ref{P:NoLiftLu}). We also prove, by a direct construction, that~$\vL$ (thus also~$\vD$) can be lifted, with respect to the functor~$\rV$, by \emph{exchange rings}. In a sense, everything in that section boils down to finding a matrix~$c$ (given in~\eqref{Eq:MatrixcExch}) satisfying the constraints given in~\eqref{Eq:cequivc0} and~\eqref{Eq:v(c)=c}.

\textbf{In Section~\ref{S:C*algLift}}, we construct a lifting of~$\vL$ by \emph{C*-algebras}. These C*-algebras have both real rank and stable rank equal to~$1$, while they have index of nilpotence~$2$. By using Theorem~\ref{T:NoLift}, it can be seen that these values are optimal.

\textbf{In Section~\ref{S:Al3}}, we take advantage of the conflict between Theorem~\ref{T:NoLift} (a non-lifting result of~$\vD$) and Proposition~\ref{P:cDliftedExch} (a lifting result of~$\vD$ by exchange rings) to construct, for every field~$K$, a unital exchange $K$-algebra~$R_K$ such that~$\rV(R_K)$, although being the positive cone of a dimension group, is never isomorphic to~$\rV(B)$ for either a C*-algebra of real rank~$0$ or a regular ring~$B$. Furthermore, $R_K$ has at most $\aleph_3+\card K$ elements, and it has index of nilpotence~$2$. In order to achieve this, we are using a fair amount of heavy machinery established earlier in the book Gillibert and Wehrung~\cite{GiWe2}. That work introduces tools of categorical algebra and infinite combinatorics, called \emph{larders} and \emph{lifters}, that are designed to turn \emph{diagram counterexamples} to \emph{object counterexamples}. These tools are effective mainly on lattice-indexed diagrams, which is the reason why we need to keep the cube~$\vD$ instead of the better looking~$\vL$. The results of~\cite{GiWe2} being now used as a toolbox, it should be possible to read the proofs of Section~\ref{S:Al3} without needing to ingest large amounts of larder technology. In particular, the ring-theoretical and C*-algebraic arguments used there are elementary.

\textbf{In Section~\ref{S:Al3Sep}}, we show how to extend to C*-algebras what we did for exchange rings in Section~\ref{S:Al3}. In particular, we construct a unital C*-algebra~$E$, of real rank and stable rank both equal to~$1$, such that~$\rV(E)$, although being the positive cone of a dimension group, is never isomorphic to~$\rV(B)$ for either a C*-algebra of real rank~$0$ or a regular ring~$B$. Furthermore, $E$ is $\aleph_3$-separable, and it has index of nilpotence~$2$.

\textbf{In Section~\ref{S:Pbs}}, we list some open problems.

We represent on Figure~\ref{Fig:RingClass} the implications between the main attributes of rings used in the paper.

\begin{figure}[htb]
 \[
\xymatrixrowsep{2pc}\xymatrixcolsep{1.5pc}
\def\labelstyle{\displaystyle}
 \xymatrix{
 & \text{Weakly V-semiprimitive} & \\
 \text{exchange}\ar@{=>}[ru] & \text{V-semiprimitive}\ar@{=>}[u] & 
 \text{semiprimitive}\\
 & \text{semiprimitive exchange}\ar@{=>}[lu]\ar@{=>}[ru]\ar@{=>}[u]
 & \text{C*-algebra}\ar@{=>}[u]\\
 & \text{regular}\ar@{=>}[u] &
 \text{C*-algebra of real rank }0\ar@{=>}[lu]\ar@{=>}[u] &
 }
 \]
\caption{Attributes of rings}
\label{Fig:RingClass}
\end{figure}
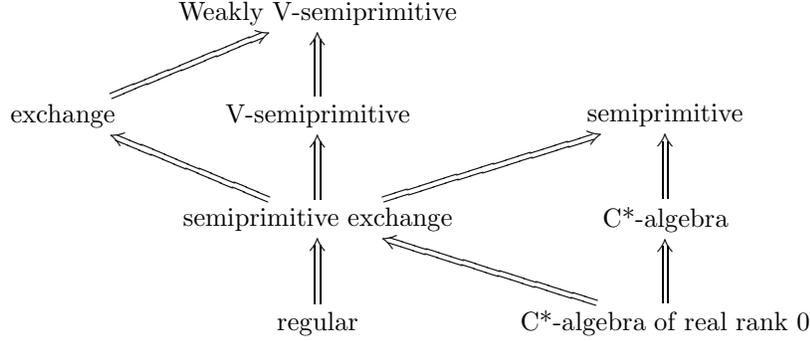

\section{Basic notions about \cm s, refinement monoids, dimension groups}\label{S:BasicMon}

We refer to Goodearl~\cite{Gpoag} for partially ordered abelian groups, interpolation groups, dimension groups. We set $\ZZ^+:=\set{0,1,2,\dots}$ and $\NN:=\ZZ^+\setminus\set{0}$.

We shall write all our \cm s additively.
A submonoid~$I$ of a \cm~$M$ is an \emph{o-ideal} of~$M$ if $x+y\in I$ implies that $x\in I$ and $y\in I$, for all $x,y\in I$. We say that~$M$ is \emph{conical} if $\set{0}$ is an o-ideal. Every \cm~$M$ can be endowed with its \emph{algebraic preordering}~$\leq$, defined by $x\leq y$ if there exists $z\in M$ such that $y=x+z$. We say that an element~$e\in M$ is an \emph{order-unit of~$M$} if for each $x\in M$ there exists $n\in\NN$ such that $x\leq ne$. For $n\in\ZZ^+$, an element $a\in M$ has \emph{index at most~$n$} if $(n+1)x\leq a$ implies that $x=0$, for each $x\in M$.

If~$I$ is an o-ideal of a \cm~$M$, the binary relation~$\equiv_I$ on~$M$ defined by
 \begin{equation}\label{Eq:M/Iequiv}
 a\equiv_Ib\ \Longleftrightarrow\ (\exists x,y\in I)(a+x=b+y)\,,\quad
 \text{for all }a,b\in M\,,
 \end{equation}
is a monoid congruence of~$M$, and $M/I:=M/{\equiv_I}$ is a conical \cm.

The \emph{Grothendieck group} of~$M$ is the initial object in the category of all monoid homomorphisms from~$M$ to a group. It consists of an abelian group~$G$ with a monoid homomorphism $\eps\colon M\to G$, and we shall always endow it with the unique compatible preordering with positive cone~$\eps(M)$ (i.e., $x\leq y$ if{f} $y-x\in\eps(M)$).

A \cm~$M$ has the \emph{refinement property}, or is a \emph{refinement monoid} (cf. Dobbertin~\cite{Dobb82}), if for all $a_0,a_1,b_0,b_1\in M$ such that $a_0+a_1=b_0+b_1$, there are $c_{j,k}\in M$ (for $j,k\in\set{0,1}$) such that $a_j=c_{j,0}+c_{j,1}$ and $b_j=c_{0,j}+c_{1,j}$ for each $j\in\set{0,1}$. A very special class of refinement monoids consists of the positive cones of \emph{dimension groups}. A partially ordered abelian group~$G$ is a dimension group if it is directed (i.e., $G=G^++(-G^+)$), unperforated (i.e., $mx\geq 0$ implies $x\geq 0$ for each $m\in\NN$ and each $x\in G$) and the positive cone~$G^+$ satisfies refinement. A \emph{simplicial group} is a group of the form~$\ZZ^n$, for a natural number~$n$, ordered componentwise. Every simplicial group is a dimension group. Conversely, every dimension group is a direct limit of simplicial groups (this follows from Effros, Handelman, and Shen~\cite{EHS}, but the form of this result established earlier in~Grillet~\cite{Gril76} is easily seen to be equivalent). We shall also call a \emph{simplicial monoid} the positive cone of a simplicial group (i.e., $(\ZZ^+)^n$ for some $n\in\ZZ^+$).

\begin{definition}\label{D:PointedMon}
For a subcategory~$\xC$ of the category~$\CM$ of conical \cm s with monoid homomorphisms, we shall denote by~$\xC(1)$ the category of \emph{monoids with order-unit} in~$\xC$: the objects of~$\xC(1)$ are the pairs $(M,a)$, where~$M$ is an object of~$\xC$ and~$a$ is an order-unit of~$M$, and a morphism from $(M,a)$ to $(N,b)$ is a monoid homomorphism $f\colon M\to N$ such that $f(a)=b$ (i.e., $f$ is \emph{normalized}).
\end{definition}

\section{Basic notions about rings: nonstable $\rK_0$-theory, semiprimitivity}\label{S:rV(R)}

\begin{definition}\label{D:CatRings1}
Not all of our rings will be unital, and consequently not all our ring homomorphisms will preserve the unit. For a subcategory~$\xR$ of the category~$\Ring$ of rings and ring homomorphisms, we shall denote by~$\xR(1)$ the subcategory of~$\xR$ consisting of all unital rings in~$\xR$, with unital ring homomorphisms, possibly with any additional structure present in~$\xR$.

For example, we shall denote by~$\xRR$ the category of all C*-algebras of real rank~$0$ with \Chom s, and by $\xRR(1)$ the category of all unital C*-algebras of real rank~$0$ with unital \Chom s. A useful subcategory of~$\xRR$ is the full subcategory~$\AF$ of all the \emph{approximately finite}, or \emph{AF}, C*-algebras. By definition, a C*-algebra is AF if it is a (possibly uncountable) C*-direct limit of finite-dimensional C*-algebras (not all our AF algebras will be separable).
\end{definition}

For basic concepts about (stable or nonstable) K-theory of rings and C*-algebras, we refer to Goodearl~\cite[Section~4]{Good94} for the unital case, Ara~\cite[Section~3]{Ara97} for the general case, Blackadar~\cite[Chapter~3]{Black98} for the case of C*-algebras. For basic notions about C*-algebras we refer to Murphy~\cite{Murp}.

For a ring~$R$ (associative but not necessarily unital), we shall often identify the ring~$\Mat_n(R)$ of all $n\times n$ matrices over~$R$ with its image, \emph{via} the embedding $x\mapsto\begin{pmatrix}x&0\\ 0&0\end{pmatrix}$, in $\Mat_{n+1}(R)$. Hence the elements of the (non-unital) ring $\Mat_\infty(R):=\bigcup_{n\in\NN}\Mat_n(R)$ can be identified with the countably infinite matrices with entries from~$R$ and only finitely many nonzero entries.

We define $\Idp_n(R)$ as the set of all idempotent elements of $\Mat_n(R)$, for each $n\in\NN\cup\set{\infty}$. For a ring homomorphism $f\colon R\to S$, we denote by $\Idp_n(f)$ the map from $\Idp_n(R)$ to $\Idp_n(S)$ defined by the rule
 \[
 \Idp_n(f)\bigl((a_{j,k})_{1\leq j,k\leq n}\bigr):=
 \bigl(f(a_{j,k})\bigr)_{1\leq j,k\leq n}\,,\quad
 \text{for each }(a_{j,k})_{1\leq j,k\leq n}\in\Idp_n(R)\,.
 \]
We shall sometimes write $f(a)$ instead of $\Idp_n(f)(a)$, for $a\in\Idp_n(R)$.

Idempotents $a,b\in\Idp_\infty(R)$ are \emph{orthogonal} if $ab=ba=0$.

The (Murray-von Neumann, algebraic) \emph{equivalence} of idempotents $a,b\in\Mat_\infty(R)$ is defined by $a\sim b$ if{f} there are $x,y\in\Mat_\infty(R)$ such that $a=xy$ and $b=yx$. Replacing~$x$ by $axb$ and $y$ by $bya$, we see that we can assume that $x=axb$ and $y=bya$. We denote by~$[a]_R$, or $[a]$ if~$R$ is understood, the $\sim$-equivalence class of~$a$. For all $a,b\in\Idp_\infty(R)$ there exists $b'\sim b$ such that $ab'=b'a=0$; it follows that equivalence classes can be added, \emph{via} the formula
 \begin{equation}\label{Eq:DefAdd[]}
 [a]+[b]:=[a+b]\text{ if }ab=ba=0\,,\text{ for all }a,b\in\Idp_\infty(R)\,.
 \end{equation}
The set $\rV(R):=\setm{[a]}{a\in\Idp_\infty(R)}$, endowed with the addition defined in~\eqref{Eq:DefAdd[]}, is a conical \cm. If~$R$ is unital, then~$[1_R]$ is an order-unit of~$\rV(R)$ and~$\rV(R)$ is isomorphic to the monoid of all isomorphism types of finitely generated projective right (resp., left) $R$-modules, with addition defined by $[X]+[Y]=[X\oplus Y]$ (where $[X]$ now denotes the isomorphism class of~$X$). We shall also set $\rVu(R):=(\rV(R),[1_R])$. 

The assignment $R\mapsto\rV(R)$ can be extended to a \emph{functor} from~$\Ring$ to~$\CM$: for a homomorphism $f\colon R\to S$ of rings (resp., of unital rings), there is a unique monoid homomorphism $\rV(f)\colon\rV(R)\to\rV(S)$ such that $\rV(f)([a]_R)=[\Idp_\infty(f)(a)]_S$ for each $a\in\Idp_\infty(R)$. The functor~$\rV$ preserves all direct limits (i.e., in categorical language, \emph{directed colimits}) and finite products: for instance, $\rV(\varinjlim_{j\in I}R_j)\cong\varinjlim_{j\in I}\rV(R_j)$ (for direct limits) and $\rV(R\times S)\cong\rV(R)\times\rV(S)$. All these facts extend to the unital case, giving a functor~$\rVu\colon\Ring(1)\to\CM(1)$, with $\rV^1(R):=(\rV(R),[1_R])$ for every unital ring~$R$.

If~$R$ is unital, $\rK_0(R)$ is defined as the (preordered) Grothendieck group of~$\rV(R)$.





For a two-sided ideal~$I$ of a ring~$R$ and the inclusion map $f\colon I\into R$, the map~$\rV(f)$ embeds~$\rV(I)$ into~$\rV(R)$. We shall thus often identify~$\rV(I)$ with its image, under~$\rV(f)$, in~$\rV(R)$. It is an \emph{o-ideal} of~$\rV(R)$ (cf. Section~\ref{S:BasicMon}).

\begin{lemma}[folklore]\label{L:VmeasDec}
Let $R$ be a ring, let $c\in\Idp_\infty(R)$, and let $\alpha,\beta\in\rV(R)$. If $[c]=\alpha+\beta$, then there are orthogonal idempotents $a,b\in\Idp_\infty(R)$ such that $c=a+b$, $[a]=\alpha$, and $[b]=\beta$.
\end{lemma}

\begin{note}
Observe, in particular, that if $c\in R$, then $a,b\in R$.
\end{note}

\begin{proof}
There are orthogonal idempotents $u,v\in\Idp_\infty(R)$ such that $[u]=\alpha$ and $[v]=\beta$. As $[u+v]=\alpha+\beta=[c]$, there are $x,y\in\Mat_\infty(R)$ such that $c=xy$ while $u+v=yx$. The matrices $a:=xuy$ and $b:=xvy$ are as required.
\end{proof}

The following result is observed, for unital exchange rings~$R$ and $E\subseteq R$, in the course of the proof of Pardo \cite[Teorema~4.1.7]{Enric}. It is contained, in full generality, in Ara and Goodearl~\cite[Proposition~10.10]{ArGo11}.

\begin{lemma}\label{L:GenV(I)}
Let~$R$ be a ring, let $E\subseteq\Idp_\infty(R)$, and denote by~$I$ the two-sided ideal of~$R$ generated by the entries of all the elements of~$E$. Then~$\rV(I)$ is the o-ideal of~$\rV(E)$ generated by $\setm{[x]_R}{x\in E}$.
\end{lemma}

A ring~$R$ has \emph{index of nilpotence at most~$n$} if $x^{n+1}=0$ implies that $x^n=0$ for all $x\in R$. It is proved in Yu~\cite[Corollary~4]{Yu95} that every unital exchange ring with finite index of nilpotence has stable rank~$1$. The following easy result gives some more information.

\begin{lemma}\label{L:BasicVcAK}
Let~$n$ be a positive integer and let~$R$ be a unital ring with index of nilpotence at most~$n$. Then~$[1_R]$ has index at most~$n$ in~$\rV(R)$. Furthermore, if~$\rV(R)$ satisfies the refinement property, then it is the positive cone of a dimension group.
\end{lemma}

\begin{proof}
Let $\xi\in\rV(R)$ such that $(n+1)\xi\leq[1]$. By Lemma~\ref{L:VmeasDec}, there are orthogonal idempotents~$e_0$, \dots, $e_n$, $a$ of~$R$ such that $1=a+\sum_{i=0}^ne_j$ while $[e_j]=\xi$ for each~$j\leq n$. For each $j<n$, let $\varphi_j\colon e_jR\to e_{j+1}R$ be an isomorphism of right $R$-modules, and denote by~$\varphi$ the unique endomorphism of~$R_R$ extending all the~$\varphi_j$ such that $\varphi(a)=\varphi(e_n)=0$. Setting $x:=\varphi(1)$, we obtain that $\varphi^{n+1}=0$, thus $x^{n+1}=0$, and thus, by assumption, $x^n=0$, so $\varphi^n=0$, and so, as~$\varphi^n(e_0R)=e_nR$, we get $e_n=0$, therefore $\xi=0$. This shows that~$[1]$ has index at most~$n$ in~$\rV(R)$.

Suppose now that~$\rV(R)$ is a refinement monoid. As~$[1]$ is an order-unit in that monoid, it follows that every element of~$\rV(R)$ has finite index (cf. Wehrung \cite[Corollary~3.12]{WDim}), and so~$\rV(R)$ is the positive cone of a dimension group (cf. Wehrung \cite[Proposition~3.13]{WDim}).
\end{proof}

In case~$R$ is a C*-algebra, each $\Mat_n(R)$, for $n\in\NN$, is also a C*-algebra, and it is often convenient to replace \emph{algebraic equivalence of idempotents} by \emph{*-equivalence of projections} (a \emph{projection} is a self-adjoint idempotent), namely,
 \[
 a\simst b\quad\text{if}\quad(\exists x)(a=xx^*\text{ and }b=x^*x)\,,\quad
 \text{for all projections }a,b\,.
 \]
This can be done, because every idempotent in a C*-algebra is equivalent to a projection (and even generates the same right ideal, cf. Kaplansky \cite[Theorem~26, p.~34]{Kapl}) and two projections are equivalent if{f} they are *-equivalent (cf. Kaplansky \cite[Theorem~27, p.~35]{Kapl}).

\begin{lemma}\label{L:IbarC*}
Let~$L$ be a left ideal in a C*-algebra~$A$. Then every projection of~$\ol{L}$ is equivalent to some projection of~$L$.
\end{lemma}

\begin{proof}
Let~$e$ be a projection of~$\ol{L}$. There exists a self-adjoint $a\in L$ such that $\norm{e-a}\leq1/3$. By R\o rdam, Larsen, and Laustsen~\cite[Lemma~2.2.3]{RLLa}, the spectrum of~$a$ is contained in $[-1/3,1/3]\cup[2/3,4/3]$. Let $f\colon\RR\to\RR$ the unique continuous function such that $f(x)=0$ for each $x\leq1/3$, $f(x)=1/x$ for each $x\geq 2/3$, and~$f$ is affine on $[1/3,2/3]$. Observe that $f(x)x=0$ for $x\in[-1/3,1/3]$ and $f(x)x=1$ for $x\in[2/3,4/3]$, thus, using continuous function calculus, $b:=f(a)a$ is a projection. It belongs to~$L$ as $a\in L$, and $\norm{a-b}\leq1/3$. Now $\norm{e-b}\leq2/3$, so $e\sim b$.
\end{proof}

\subsection*{Semiprimitive rings} It is known (cf. Jacobson~\cite[Theorem~1]{Jaco45}, Herstein \cite[Theorem~1.2.3]{Hers}) that the Jacobson radical~$\rJ(R)$ of a ring (not necessarily unital)~$R$ is the largest two-sided ideal~$J$ of~$R$ such that $(\forall x\in J)(\exists s\in R)(x+s-sx=0)$. This concept is left-right symmetric. A ring~$R$ is \emph{semiprimitive} (\emph{semi-simple} in Jacobson~\cite[Section~2]{Jaco45} or Herstein~\cite[page~16]{Hers}, but we are using the seemingly more common terminology) if $\rJ(R)=\set{0}$. Not every exchange ring is semiprimitive (consider the upper triangular matrices over a field, cf. Example~\ref{Ex:TrMat}).




The following is contained in Jacobson~\cite[Theorem~8]{Jaco45}.

\begin{proposition}\label{P:RegSemipr}
Every regular ring is semiprimitive.
\end{proposition}

The following is contained in Dixmier~\cite[Th\'eor\`eme~2.9.7]{Dixm}.

\begin{proposition}\label{P:C*Semipr}
Every C*-algebra is semiprimitive.
\end{proposition}
%


\section{Exchange rings}\label{S:BasicExch}

\emph{Exchange rings} have been first defined in the unital case in Warfield~\cite{Warf72}, then in the general case by Ara~\cite{Ara97}. A ring~$R$ is an exchange ring if for all $x\in R$, there are an idempotent $e\in R$ and $r,s\in R$ such that $e=rx=x+s-sx$. This condition is left-right symmetric. Not every exchange ring is a two-sided ideal in a unital exchange ring, see \cite[Example~4, page~412]{Ara97}. A ring~$R$ is (von~Neumann) \emph{regular} if for all $x\in R$, there exists $y\in R$ such that $xyx=x$. Every regular ring is an exchange ring, and the converse fails (cf. Example~\ref{Ex:TrMat}).

The following result is established in the unital case in Ara, Goodearl, O'Meara, and Pardo~\cite[Theorem~7.2]{AGOP}, and then extended to the non-unital case in Ara~\cite[Theorem~3.8]{Ara97}.

\begin{proposition}\label{P:RR0Exch}
A C*-algebra has real rank~$0$ if{f} it is an exchange ring.
\end{proposition}

The following is contained in Proposition~1.3 and Theorem~1.4 in Ara~\cite{Ara97}.

\begin{lemma}\label{L:eReErch}
The following statements hold, for any exchange ring~$R$.
\begin{enumerate}
\item $eRe$ is an exchange ring, for any idempotent~$e\in R$.

\item $\Mat_n(R)$ is an exchange ring, for any positive integer~$n$.
\end{enumerate}
\end{lemma}

It is well-known that Lemma~\ref{L:eReErch} extends to C*-algebras of real rank~$0$ (cf. Corollary~2.8 and Theorem~2.10 in Brown and Pedersen~\cite{BrPe}).

The following result is proved in \cite[Proposition~1.5]{Ara97}.

\begin{proposition}\label{P:V(Exch)}
The monoid~$\rV(R)$ has refinement, for every exchange ring~$R$.
\end{proposition}

The first part of the following result is proved in Ara~\cite[Theorem~2.2]{Ara97}. The second part is proved, in the unital case, in Ara, Goodearl, O'Meara, and Pardo~\cite[Proposition~1.4]{AGOP}. The proof can be trivially extended to the non-unital case, by using Lemma~\ref{L:eReErch}.

\begin{proposition}\label{P:V(R/I)}
Let~$I$ be a two-sided ideal in a ring~$R$. Then~$R$ is an exchange ring if{f} $I$ and $R/I$ are both exchange rings and idempotents can be lifted modulo~$I$. Furthermore, if~$R$ is an exchange ring, then the canonical homomorphism $\rV(R)\to\rV(R/I)$ is surjective, and it induces an isomorphism $\rV(R)/\rV(I)\to\rV(R/I)$.
\end{proposition}

In the statement of Proposition~\ref{P:V(R/I)}, $\rV(I)$ is identified with its image in~$\rV(R)$, which is an o-ideal of~$\rV(R)$. The quotient $\rV(R)/\rV(I)$ is defined in~\eqref{Eq:M/Iequiv} and the comment following.





{}From now on we shall denote by $\Reg$ (resp., $\Exch$) the category of all regular rings (resp., exchange rings) with ring homomorphisms.

\section{Weakly V-semiprimitive rings; an unliftable diagram}\label{S:FirstDiagr}

\begin{definition}\label{D:perp}
Elements~$\alpha$ and~$\beta$ in a conical \cm~$M$ are \emph{orthogonal}, in notation $\alpha\perp\beta$, if the following statement holds:
 \[
 (\forall n\in\NN)(\forall\gamma\in M)\bigl(
 (\gamma\leq n\alpha\text{ and }\gamma\leq n\beta)
 \Rightarrow\gamma=0\bigr)\,.
 \]
Equivalently, the o-ideals generated by~$\alpha$ and~$\beta$, respectively, meet in~$\set{0}$.
\end{definition}

\begin{lemma}\label{L:V(A)capV(B)}
Let~$R$ be an exchange ring, let $a,b\in\Idp_\infty(R)$ with $[a]_R\perp[b]_R$, and denote by~$A$ and~$B$ the two-sided ideals of~$R$ generated by the entries of~$a$ and~$b$, respectively. Then $A\cap B$ is contained in~$\rJ(R)$.
\end{lemma}

\begin{proof}
Set $I:=A\cap B$. It follows from Lemma~\ref{L:GenV(I)} that for each $e\in\Idp_\infty(I)$, there exists $n\in\NN$ such that $[e]\leq n[a]$ and $[e]\leq n[b]$. As $[a]\perp[b]$, it follows that $e=0$. Hence $\rV(I)=\set{0}$. As~$I$ is an exchange ring without nonzero idempotents, it must be contained in~$\rJ(R)$.
\end{proof}

\begin{definition}\label{D:wVsem}
A ring~$R$ is \emph{weakly V-semiprimitive} if for all $a,b\in\Idp_\infty(R)$, the relation $[a]_R\perp\nobreak[b]_R$ implies that $ab\in\rJ(\Mat_\infty(R))$.
\end{definition}

As an immediate consequence of Lemma~\ref{L:V(A)capV(B)}, we record the following.

\begin{proposition}\label{P:Exch2wVsem}
Every exchange ring is weakly V-semiprimitive.
\end{proposition}

The converse of Proposition~\ref{P:Exch2wVsem} fails: not every weakly V-semiprimitive ring is an exchange ring (e.g., the ring~$\ZZ$ of all integers).

Consider the diagram~$\vK$ represented in the left hand side of Figure~\ref{Fig:OneLoop}: by definition, $\bs(x,y)=(y,x)$ and $\bh(x,y)=x+y$, for all $(x,y)\in(\ZZ^+)^2$.

\begin{figure}[htb]
 \[
\xymatrixrowsep{2pc}\xymatrixcolsep{1.5pc}
\def\labelstyle{\displaystyle}
 \xymatrix{
 (\ZZ^+)^2\lloopd{15}{7}_{\bs}\ar[rr]^{\bh} && \ZZ^+ &&&&
 R\lloopd{15}{7}_{s}\ar[rr]^h && S
 }
 \]
\caption{The diagram $\vK$ and a lifting~$\vR$ of $\vK$}
\label{Fig:OneLoop}
\end{figure}

The diagram~$\vK$ and Theorem~\ref{T:FirstNonLift} were both suggested to the author by Ken Goodearl, thus improving the author's original formulation of Corollary~\ref{C:FirstNonLift} (which was stated for functors $\lift\colon\xS(1)\to\nobreak\Exch$).

\begin{theorem}\label{T:FirstNonLift}
The diagram~$\vK$ has no lifting, with respect to the functor~$\rV$, by any diagram labeled as on the right hand side of Figure~\textup{\ref{Fig:OneLoop}} in such a way that~$R$ is weakly V-semiprimitive and $h\circ s=h$.
\end{theorem}

\begin{proof}
Let $\eta\colon\rV\vR\to\vK$ be a natural equivalence, with component isomorphisms $\eta_R\colon\rV(R)\to\nobreak(\ZZ^+)^2$ and $\eta_S\colon\rV(S)\to\ZZ^+$. We obtain a commutative diagram as on Figure~\ref{Fig:UnfoldingLoop}.

\begin{figure}[htb]
 \[
\xymatrixrowsep{2pc}\xymatrixcolsep{1.5pc}
\def\labelstyle{\displaystyle}
 \xymatrix{
 \rV(R)\ar[rr]^{\rV(s)}\ar[d]_{\eta_R} &&
 \rV(R)\ar[rr]^{\rV(h)}\ar[d]_{\eta_R} && \rV(S)\ar[d]_{\eta_S}\\
(\ZZ^+)^2\ar[rr]^{\bs} &&(\ZZ^+)^2\ar[rr]^{\bh} &&\ZZ^+
 }
 \]
\caption{A commutative diagram of \cm s}
\label{Fig:UnfoldingLoop}
\end{figure}

There are $a,b\in\Idp_\infty(R)$ such that $\eta_R([a]_R)=(1,0)$ and $\eta_R([b]_R)=(0,1)$. As~$\eta_R$ is an isomorphism, $[a]_R\perp[b]_R$. Chasing around the diagram of Figure~\ref{Fig:UnfoldingLoop}, we get
 \[
 \eta_R([b]_R)=(0,1)=\bs(1,0)=(\bs\circ\eta_R)([a]_R)=
 (\eta_R\circ\rV(s))([a]_R)=\eta_R([s(a)]_R)\,,
 \]
thus $[b]_R=[s(a)]_R$, and thus $[a]_R\perp[s(a)]_R$, and so, as~$R$ is weakly V-semiprimitive, $as(a)\in\rJ(\Mat_\infty(R))$. Hence there exists $x\in\Mat_\infty(R)$ such that $as(a)+x-as(a)x=0$. By applying~$h$ to each side of this equation and observing that~$a^2=a$ while $h\circ s=h$, we obtain that $h(a)+h(x)-h(a)h(x)=0$, thus, multiplying this equation on the left by~$h(a)$, we get $h(a)=0$. Therefore,
 \[
 1=\bh(1,0)=(\bh\circ\eta_R)([a]_R)=(\eta_S\circ\rV(h))([a]_R)=
 \eta_S([h(a)]_S)=0\quad\text{in }\ZZ^+\,,
 \]
a contradiction.
\end{proof}

\begin{corollary}\label{C:FirstNonLift}
Denote by~$\xS(1)$ the category of all simplicial monoids with order-unit and normalized monoid homomorphisms, and by $\wVSem$ the category of all weakly V-semiprimitive rings and ring homomorphisms.
Then there is no functor $\lift\colon\xS(1)\to\nobreak\wVSem$ such that $\rV\circ\lift\cong\id$.
\end{corollary}

We shall see in Proposition~\ref{P:vecCLift} that while the diagram~$\vK$ has no lifting, with respect to the functor~$\rV$, by weakly V-semiprimitive rings with $h\circ s=h$, it has one by unital C*-algebras.

\section{{}From a transfinite counterexample to non-splitting of~$\rK_0$}\label{S:NoSplit}

As the following result will be superseded by Corollary~\ref{C:NoLift3}, we shall only outline its proof.

\begin{proposition}\label{P:NonSplit}
Denote by~$\xSe_{\mathrm{grp}}(1)$ the category of all simplicial groups with order-unit with normalized positive group embeddings. Then there is no functor\linebreak $\lift\colon\xSe_{\mathrm{grp}}(1)\to\AF$ such that $\rK_0\circ\lift\cong\id$.
\end{proposition}

\begin{proof}[Outline of proof]
We shall first invoke some ideas from universal algebra and lattice theory (cf. McKenzie, McNulty, and Taylor~\cite{MMTa87} for background). A \emph{bounded lattice} is a lattice with smallest element and largest element (usually denoted by~$0$ and~$1$, respectively). A \emph{variety of bounded lattices} is the class of all bounded lattices that satisfy a given set of identities, written in the language $(\vee,\wedge,0,1)$. For a variety~$\xV$ of bounded lattices and a set~$X$, we denote by $\rF_{\xV}(X)$ the \emph{free object} on~$X$ within the variety~$\xV$. We shall denote by~$\sM_3$ the lattice of length two with three atoms (cf. Figure~\ref{Fig:M3}) and by~$\xM_3$ the variety that it generates.
\begin{figure}[htb]
\includegraphics{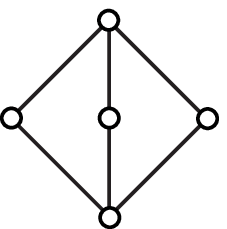}
\caption{The lattice $\sM_3$}\label{Fig:M3}
\end{figure}
As~$\sM_3$ is finite, the variety~$\xM_3$ is \emph{locally finite}, that is, $\rF_{\xM_3}(X)$ is finite whenever~$X$ is finite (this is a classical argument of universal algebra, see, for example, the proof of
McKenzie, McNulty, and Taylor~\cite[Lemma~4.98]{MMTa87}). Furthermore,
 \begin{equation}\label{Eq:FM3(X)lim}
 \rF_{\xM_3}(X)=\varinjlim
 \famm{\rF_{\xM_3}(Y)}{Y\subseteq X\text{ finite}}\,,
 \end{equation}
with canonical transition maps and limiting maps, for any set~$X$.

We proceed by invoking a lattice-theoretical tool. The \emph{dimension monoid}~$\Dim L$ of a lattice~$L$, defined in Wehrung~\cite[Definition~1.1]{WDim}, is always a conical \cm; furthermore, the assignment $L\mapsto\Dim L$ extends to a \emph{functor}. As in~\cite{WDim}, we define~$\rK_0(L)$ as the (preordered) Grothendieck group of~$\Dim L$.

Now fix any set~$X$ with at least~$\aleph_2$ elements. By \cite[Corollary~10.30]{WDim}, there is no regular ring~$R$ such that $\Dim\rF_{\xM_3}(X)\cong\rV(R)$. Furthermore, as the~$\Dim$ functor preserves direct limits (cf. \cite[Proposition~1.4]{WDim}), it follows from~\eqref{Eq:FM3(X)lim} that
 \begin{equation}\label{Eq:DimFM3(X)lim}
 \Dim\rF_{\xM_3}(X)=\varinjlim
 \famm{\Dim\rF_{\xM_3}(Y)}{Y\subset X\text{ finite}}\,,
 \end{equation}
and thus
 \begin{equation}\label{Eq:K0FM3(X)lim}
 \rK_0\bigl(\rF_{\xM_3}(X)\bigr)=\varinjlim
 \famm{\rK_0\bigl(\rF_{\xM_3}(Y)\bigr)}{Y\subset X\text{ finite}}\,.
 \end{equation}
For each finite $Y\subset X$, the lattice~$\rF_{\xM_3}(Y)$ is finite modular, thus its dimension monoid $\Dim\rF_{\xM_3}(Y)$ is a finitely generated simplicial monoid (cf. \cite[Proposition~5.5]{WDim}); in particular, it embeds into its Grothendieck group $\rK_0\bigl(\rF_{\xM_3}(X)\bigr)$.
Furthermore, $\rF_{\xM_3}(Y)$ is a retract of~$\rF_{\xM_3}(X)$ (send every element of~$X\setminus Y$ to~$0$), thus $\rK_0\bigl(\rF_{\xM_3}(Y)\bigr)$ is a retract of $\rK_0\bigl(\rF_{\xM_3}(X)\bigr)$. Finally, all the transition morphisms in~\eqref{Eq:DimFM3(X)lim} are trivially seen to preserve the canonical order-units (namely, using the notation of~\cite{WDim}, $\Delta(0,1)$). Therefore, the direct limit in~\eqref{Eq:K0FM3(X)lim} is, in fact, a \emph{direct union of dimension groups with order-unit}. Hence, by assumption on~$\lift$, we can define an AF C*-algebra as a C*-direct limit,
 \begin{equation}\label{Eq:AFCX}
 A:=\varinjlim\famm{\lift\rK_0\bigl(\rF_{\xM_3}(Y)\bigr)}
 {Y\subset X\text{ finite}}\quad\text{for C*-algebras}\,.
 \end{equation}
Let~$R$ be a dense locally matricial algebra (over~$\CC$) in~$A$. Then $\rK_0(R)\cong\rK_0(A)$. Furthermore, as the functor~$\rK_0$ preserves C*-direct limits and $\rK_0\circ\lift\cong\id$, it follows from~\eqref{Eq:AFCX} that
 \[
 \rK_0(A)=\varinjlim\famm{\rK_0\bigl(\rF_{\xM_3}(Y)\bigr)}
 {Y\subset X\text{ finite}}\quad\text{in the category of dimension groups}\,,
 \]
so $\rK_0(A)\cong\rK_0\bigl(\rF_{\xM_3}(X)\bigr)$, and so $\rV(R)\cong\Dim\rF_{\xM_3}(X)$. However, as~$R$ is von~Neumann regular, this is a contradiction.
\end{proof}

The proof of Proposition~\ref{P:NonSplit} may feel quite outlandish to a number of readers: while its statement is, essentially, combinatorial, its proof involves the counterexample of cardinality~$\aleph_2$ introduced in~\cite{NonMeas}. Extracting the gist of that proof will lead us to the proof of Theorem~\ref{T:NoLift}, which is, to a large extent, a finitary result, involving neither transfinite cardinals, nor universal algebra, nor lattice theory. It will also extend Proposition~\ref{P:NonSplit} to C*-algebras of real rank~$0$.

\section{The category of premeasured rings}\label{S:MeasRing}

\begin{definition}\label{D:Meas}
Let $M$ be a \emph{conical} \cm\ and let~$R$ be a ring. An \emph{$M$-valued premeasure on~$R$} is a map~$\mu\colon\Idp_\infty(R)\to M$ such that
\begin{itemize}
\item[(M0)] $\mu(0)=0$;

\item[(M1)] $\mu(a+b)=\mu(a)+\mu(b)$, for all orthogonal idempotents $a,b\in\Idp_\infty(R)$;

\item[(M2)] $a\sim b$ implies that $\mu(a)=\mu(b)$, for all $a,b\in\Idp_\infty(R)$. (\emph{Recall that~$\sim$ stands for Murray-von Neumann equivalence, see Section~\textup{\ref{S:rV(R)}}}.)
\end{itemize}
We say that the premeasure~$\mu$ is a \emph{measure} if it satisfies the condition
\begin{itemize}
\item[(M0$^+$)] $\mu(e)=0$ if{f} $e=0$, for each $e\in\Idp_\infty(R)$.
\end{itemize}
We say that the premeasure~$\mu$ is a \emph{V-premeasure} if it satisfies the following \emph{V-condition}:
\begin{itemize}
\item[(M3)] For all $c\in\Idp_\infty(R)$ and all $\alpha,\beta\in M$ such that $\mu(c)= \alpha+\beta$, there are orthogonal idempotent $a,b\in\Idp_\infty(R)$ such that $c=a+b$, $\mu(a)= \alpha$, and $\mu(b)=\beta$.
\end{itemize}

We say that~$\mu$ is a \emph{V-measure} if it is both a measure and a V-premeasure.

We shall use the notation $M=M_{\mu}$ (the \emph{codomain} of~$\mu$).
\end{definition}

A fundamental example of a V-measure is the following.

\begin{example}\label{Ex:CanVMeas}
Let $R$ be a ring. Then the assignment $\tau_R\colon\Idp_\infty(R)\to\rV(R)$, $e\mapsto[e]_R$ is a measure on~$R$. Due to Lemma~\ref{L:VmeasDec}, $\tau_R$ is actually a V-measure. We call it the \emph{canonical V-measure} on~$R$.
\end{example}

\begin{definition}\label{D:Vhom}
Let $M$ and $N$ be conical \cm s. A homomorphism $f\colon M\to N$ is a \emph{pre-\Vhom} if whenever $a,b\in N$ and $c\in M$ with $f(c)=a+b$, there are $x,y\in M$ such that $c=x+y$, $f(x)=a$, and $f(y)=b$.

If, in addition, $f^{-1}\set{0_N}={0_M}$, then we say that~$f$ is a \emph{\Vhom} (cf. Dobbertin \cite[Definition~1.2]{Dobb82}).
\end{definition}

The following result relates the definition of a measure given above, the canonical V-measure, and the notion of \Vhom. It expresses the universality of~$\tau_R$ as a premeasure on~$R$, and its proof is a straightforward exercise.

\begin{proposition}\label{P:FactMeas}
Let~$R$ be a ring, let~$M$ be a conical \cm, and let~$\mu$ be an $M$-valued premeasure on~$R$. Then there exists a unique monoid homomorphism $\ol{\mu}\colon\rV(R)\to M$ such that $\mu=\ol{\mu}\circ\tau_R$. Furthermore,
\begin{enumerate}
\item $\mu$ is a measure if{f} $\ol{\mu}^{-1}\set{0}=\set{0}$;

\item $\mu$ is a V-premeasure if{f} $\ol{\mu}$ is a pre-\Vhom.
\end{enumerate}
\end{proposition}

\begin{definition}\label{D:MeasCat}
For a category~$\xR$ of rings and ring homomorphisms, we shall denote by $\xR^{\pmeas}$ the category where the objects (\emph{premeasured rings}) are the pairs $(R,\mu)$, where~$\mu$ is a premeasure on the ring~$R\in\xR$ (with values in some conical \cm), and, for objects $(R,\mu)$ and $(S,\nu)$, a \emph{morphism} from $(R,\mu)$ to $(S,\nu)$ is a pair $(f,\tilde{f})$, where $f\colon R\to S$ is a ring homomorphism, $\tilde{f}\colon M_{\mu}\to M_{\nu}$ is a monoid homomorphism, and $\nu\circ\Idp_\infty(f)=\tilde{f}\circ\mu$ (cf. Figure~\ref{Fig:MeasHom}).
\begin{figure}[htb]
 \[
\xymatrixrowsep{2pc}\xymatrixcolsep{1.5pc}
\def\labelstyle{\displaystyle}
\xymatrix{
R\ar[rr]_{f} && S && \Idp_\infty(R)\ar[rr]_{\Idp_\infty(f)}\ar[d]_{\mu} 
&& \Idp_\infty(S)\ar[d]^{\nu}\\
M_{\mu}\ar[rr]^{\tilde{f}} && M_{\nu}&& M_{\mu}\ar[rr]^{\tilde{f}} && M_{\nu}
}
 \]
\caption{A morphism from $(R,\mu)$ to $(S,\nu)$}
\label{Fig:MeasHom}
\end{figure}
The composition of morphisms is defined componentwise (i.e., $(f,\tilde{f})\circ(g,\tilde{g})=(f\circ g,\tilde{f}\circ\tilde{g})$).

We define similarly $\xR^{\meas}$, $\xR^{\Vpmeas}$, $\xR^{\Vmeas}$, using measures, V-premeasures, and V-measures, respectively, instead of premeasures.
\end{definition}

Next, we introduce two functors: one \emph{to} $\Ring^{\pmeas}$, and one \emph{from} $\Ring^{\pmeas}$.

\begin{definition}\label{D:TwoFunct}
The \emph{canonical functor} $\tau\colon\Ring\to\Ring^{\pmeas}$ is defined by $\tau(R):=(R,\tau_R)$ (cf. Example~\ref{Ex:CanVMeas}) and $\tau(f):=(f,\rV(f))$, for any rings~$R$ and~$S$ and any ring homomorphism $f\colon R\to S$.

The \emph{projection functor} $\forg\colon\Ring^{\pmeas}\to\CM$ is defined by $\forg(R,\mu):=M_{\mu}$ and $\forg(f,\tilde{f}):=\tilde{f}$, for any premeasured rings $(R,\mu)$ and $(S,\nu)$ and any morphism $(f,\tilde{f})\colon(R,\mu)\to(S,\nu)$.
\end{definition}

\begin{remark}\label{Rk:picirctau}
It is trivial that $\rV= \forg\circ\,\tau$. In particular, if a diagram can be lifted by \emph{rings} with respect to the functor~$\rV$, then it can be lifted by \emph{premeasured rings} with respect to the functor~$\forg $.
\end{remark}

{}From now on, we shall often omit the second coordinate~$\tilde{f}$ in a morphism $(f,\tilde{f})$ from $\Ring^{\pmeas}$, thus simply writing $f\colon(R,\mu)\to(S,\nu)$ and specifying what is~$\tilde{f}$ in case needed.

\begin{lemma}\label{L:ReflSubcat}
The category $\Exch^{\meas}$ is a reflective subcategory of~$\Exch^{\pmeas}$, with reflector $\gQr\colon\Exch^{\pmeas}\to\Exch^{\meas}$ satisfying $\forg\circ\gQr=\forg$. Furthermore,
\begin{enumerate}
\item $\gQr$ sends surjective premeasures to surjective measures;

\item $\gQr$ sends V-premeasures to V-measures;

\item $\gQr$ sends regular rings to regular rings.
\end{enumerate}
\end{lemma}

\begin{proof}
Let $\mu$ be a premeasure, with values in a conical \cm~$M$, on an exchange ring~$R$, and denote by $\ol{\mu}\colon\rV(R)\to\nobreak M$ the canonical monoid homomorphism.
Denote by~$I$ the two-sided ideal of~$R$ generated by the entries of all $x\in\Idp_\infty(R)$ such that $\mu(x)=0$.
By Lemma~\ref{L:GenV(I)} together with the conicality of~$M$, $\mu(e)=0$ for each $e\in\Idp_\infty(I)$; so $\ol{\mu}\res_{\rV(I)}=0$. Therefore, denoting by $\rho\colon\rV(R)\onto\rV(R)/\rV(I)$ the canonical projection, there exists a unique monoid homomorphism $\tilde{\mu}\colon\rV(R)/\rV(I)\to M$ such that $\ol{\mu}=\tilde{\mu}\circ\rho$. Then $R^*:=R/I$ is an exchange ring, while, composing~$\tilde{\mu}$ with the canonical isomorphism $\rV(R/I)\cong\rV(R)/\rV(I)$ (cf. Proposition~\ref{P:V(R/I)}) and using the fact that idempotents can be lifted modulo~$I$, we obtain a (necessarily unique) monoid homomorphism $\ol{\mu}^*\colon\rV(R/I)\to M$ such that
 \[
 \ol{\mu}^*([e+\Mat_\infty(I)]_{R/I})=\mu(e)\,,\quad
 \text{for each }e\in\Idp_\infty(R)\,.
 \]
The map~$\mu^*$ defined by the rule $\mu^*(e):=\ol{\mu}^*([e]_{R^*})$ for each $e\in\Idp_\infty(R^*)$ is an $M$-valued premeasure on~$R^*$, and
 \begin{equation}\label{Eq:Defnmu*}
 \mu^*(e+\Mat_\infty(I))=\mu(e)\,,\quad
 \text{for each }e\in\Idp_\infty(R)\,.
 \end{equation}
For each $e\in\Idp_\infty(R)$ such that $\mu^*(e+\Mat_\infty(I))=0$, that is, $\mu(e)=0$, it follows from the definition of~$I$ that $e\in\Mat_\infty(I)$. Therefore, \emph{$\mu^*$ is an $M$-valued measure on~$R^*$}.

\begin{sclaim}
$(R^*,\mu^*)$ is the $\Exch^{\meas}$-reflection of $(R,\mu)$, with $\Exch^{\meas}$-reflection morphism the canonical projection $\pi\colon R\onto R^*$ with $\tilde{\pi}:=\id_M$.
\end{sclaim}

\begin{scproof}
We must prove that for every measured exchange ring $(S,\nu)$ and every morphism $\varphi\colon(R,\mu)\to(S,\nu)$, there exists a unique morphism\linebreak $\varphi^*\colon(R^*,\mu^*)\to(S,\nu)$ such that $\varphi=\varphi^*\circ\pi$. For each $x\in I$, $\nu(\varphi(x))=\tilde{\varphi}(\mu(x))=\tilde{\varphi}(0)=0$, thus, as~$\nu$ is a measure, $\varphi(x)=0$. It follows that there exists a unique ring homomorphism $\varphi^*\colon R^*\to S$ such that $\varphi=\varphi^*\circ\pi$. Setting $\tilde{\varphi^*}:=\tilde{\varphi}$, we obtain that~$\varphi$ is as desired.
\end{scproof}

Denote by $\gQr\colon(R,\mu)\mapsto(R^*,\mu^*)$ the $\Exch^{\meas}$-reflection functor. As~$\mu$ and~$\mu^*$ have the same codomain and $\tilde{\varphi^*}:=\tilde{\varphi}$ in the proof of the Claim above, $\forg\circ\gQr=\gQr$ and the surjectivity of~$\mu$ implies the one of~$\mu^*$. If~$\mu$ is a V-premeasure, then, as idempotents can be lifted modulo~$I$, so is~$\mu^*$ (cf.~\eqref{Eq:Defnmu*}). Finally, if~$R$ is a regular ring, then $R^*=R/I$ is also regular.
\end{proof}

The following analogue of Lemma~\ref{L:ReflSubcat} for C*-algebras of real rank~$0$ is valid.

\begin{lemma}\label{L:ReflSubRR0}
The category $\xRR^{\meas}$ is a reflective subcategory of~$\xRR^{\pmeas}$, with reflector $\gQa\colon\xRR^{\pmeas}\to\xRR^{\meas}$ satisfying $\forg\circ\gQa=\forg$. Furthermore,
\begin{enumerate}
\item $\gQa$ sends surjective premeasures to surjective measures;

\item $\gQa$ sends V-premeasures to V-measures.
\end{enumerate}
\end{lemma}

\begin{proof}
Using Proposition~\ref{P:RR0Exch}, the proof is almost identical to the one of Lemma~\ref{L:ReflSubcat}. The main difference is the need to replace the ideal~$I$ of~$R$ by its topological closure~$\ol{I}$ (so $R/{\ol{I}}$ is still a C*-algebra, necessarily of real rank~$0$, see Brown and Pedersen~\cite[Theorem~3.14]{BrPe}). We need to prove that $\mu(x)=0$ for each projection~$x$ of~$\ol{I}$. By Lemma~\ref{L:IbarC*}, $x$ is equivalent to a projection~$y$ of~$I$. By the proof of Lemma~\ref{L:ReflSubcat}, $\mu(y)=0$; thus $\mu(x)=0$. Now the reflection morphism $\gQa\colon\xRR^{\pmeas}\to\xRR^{\meas}$ is defined by $\gQa(R,\mu):=(R^*,\mu^*)$, where this time
 \begin{equation}\label{Eq:Defnmu*RR0}
 \mu^*(e+\Mat_\infty(\ol{I}))=\mu(e)\,,\quad
 \text{for each }e\in\Idp_\infty(R)\,,
 \end{equation}
while $\pi\colon R\onto R/{\ol{I}}$ is the canonical projection and $\tilde{\pi}=\id_M$. The rest of the proof runs as the one of Lemma~\ref{L:ReflSubcat}.
\end{proof}

\section{V-semiprimitive rings}\label{S:Vsemi}

The following definition introduces a strengthening of the weakly V-semiprimitive rings introduced in Definition~\ref{D:wVsem}.

\begin{definition}\label{D:Vsemi}
A ring~$R$ is \emph{V-semiprimitive} if for all $a,b\in\Idp_\infty(R)$, $[a]_R\perp[b]_R$ in~$\rV(R)$ implies that $ab=0$.
\end{definition}

If~$\rV(R)$ is totally ordered with respect to its algebraic preordering, then~$R$ is V-semiprimitive. In particular, this holds if $\rV(R)\cong\ZZ^+$. This isomorphism occurs in case~$R$ is a principal ideal domain. Due to the Quillen-Suslin Theorem, it also holds in case~$R$ is a polynomial ring over a field. Another important class of V-semiprimitive rings, which also explains the terminology, is provided by the following result.

\begin{proposition}\label{P:SeExVsem}
Every semiprimitive exchange ring is V-semiprimitive.
\end{proposition}

We emphasize here that the rings in question are not necessarily unital.

\begin{proof}
Let~$R$ be a semiprimitive exchange ring and let $a,b\in\Idp_\infty(R)$ such that $[a]\perp\nobreak[b]$ in~$\rV(R)$. Denoting by~$A$ and~$B$ the two-sided ideals of~$R$ generated by the entries of~$a$ and~$b$, respectively, it follows from Lemma~\ref{L:V(A)capV(B)} that $A\cap B$ is contained in~$\rJ(R)$, thus, by assumption, $A\cap B=\set{0}$, so $ab=0$.
\end{proof}

The following two examples show that none of the assumptions of semiprimitivity and exchange property can be dispensed with in the statement of Proposition~\ref{P:SeExVsem}.

\begin{example}\label{Ex:TrMat}
Let $K$ be a field (any unital exchange ring would do). Then the ring~$R$ of all upper triangular $2\times2$ matrices over~$K$ is an exchange ring, but it is not V-semiprimitive. To show the latter statement, set $a:=\begin{pmatrix}1&1\\0&0\end{pmatrix}$ and $b:=\begin{pmatrix}0&0\\0&1\end{pmatrix}$. Observe that $\rV(R)\cong\ZZ^+\times\ZZ^+$, and once the two monoids are identified, $[a]=(1,0)$ while $[b]=(0,1)$. In particular, $[a]\perp[b]$, although $ab=\begin{pmatrix}0&1\\0&0\end{pmatrix}$ is nonzero.
\end{example}

\begin{example}\label{Ex:semiprnotV}
We denote by $\rC(X,A)$ the C*-algebra of all $A$-valued continuous functions on a topological space~$X$, for any C*-algebra~$A$.
We denote by $[0,1]$ the unit interval of~$\RR$ and we set
 \[
 \DD:=
 \setm{x\in\rC([0,1],\Mat_2(\CC))}{x(0)\text{ is a diagonal matrix}}\,.
 \]
This type of construction appears in Su~\cite[Section~1.2]{Su95}.
The ring~$\DD$ is a key ingredient in the main construction in Section~\ref{S:C*algLift}. Then~$\DD$ is a semiprimitive unital ring (because it is a unital C*-algebra, cf. Proposition~\ref{P:C*Semipr}). However, $\DD$ is not V-semiprimitive. Indeed, denoting by $z\colon[0,1]\into\CC$ the inclusion map and setting
 \[
 a:=\begin{pmatrix}1&z\\ 0&0\end{pmatrix}\quad\text{and}\quad
 b:=\begin{pmatrix}0&0\\ 0&1\end{pmatrix}\,,
 \]
then~$a$ and~$b$ are both idempotent and $ab=\begin{pmatrix}0&z\\ 0&0\end{pmatrix}$ is nonzero. However, we prove in Section~\ref{S:C*algLift} that $\rV(\DD)\cong\ZZ^+\times\ZZ^+$, \emph{via} an isomorphism that is easily seen to send~$[a]$ to $(1,0)$ and~$b$ to~$(0,1)$. In particular, $[a]\perp[b]$ in~$\rV(\DD)$. Therefore, $\DD$ is not V-semiprimitive.
\end{example}

\begin{lemma}\label{L:Bracklift}
Let $R$ be a V-semiprimitive ring, let~$M$ be a conical \cm, and let $\mu\colon\Idp_\infty(R)\to M$ be a measure. Let $a,b,e\in\Idp_\infty(R)$ with $a\leq e$ and $b\leq e$ while $\mu(a)\perp\mu(e-b)$ and $\mu(b)\perp\mu(e-a)$ in~$M$. Then $a=b$.
\end{lemma}

\begin{proof}
Denote by $\ol{\mu}\colon\rV(R)\to M$ the unique \Vhom\ such that $\mu(x)=\ol{\mu}([x])$ for each $x\in\Idp_\infty(R)$ (cf. Proposition~\ref{P:FactMeas}). We claim that $[a]\perp[e-b]$ in~$\rV(R)$. Indeed, otherwise there would exist a nonzero $x\in\Idp_\infty(R)$ such that $[x]\leq[a]$ and $[x]\leq[e-b]$. By applying~$\ol{\mu}$ to those inequalities, we obtain $\mu(x)\leq\mu(a)$ and $\mu(x)\leq\mu(e-b)$, a contradiction as $\mu(x)\neq0$ and $\mu(a)\perp\mu(e-b)$, and thus proving our claim. As~$R$ is V-semiprimitive, it follows that $a(e-b)=(e-b)a=0$, thus, as $a\leq e$, we get that $a=ab=ba$. By symmetry, $a=ab=ba=b$.
\end{proof}

\section{More unliftable diagrams of simplicial monoids}\label{S:Unlift}

The main combinatorial object underlying the present paper is a commutative diagram, which we shall denote by~$\vD$, of simplicial monoids and monoid homomorphisms. This diagram is indexed by the \emph{cube}, that is, the partially ordered set of all subsets of a three-element set. Its vertices are the \cm s $\bA=\bB=\ZZ^+$ and $\bA_j=\bB_j=\ZZ^+\times\ZZ^+$ ($j\in\set{0,1,2}$).
The monoid homomorphisms labeling its edges are the identity map, together with the maps $\be\colon\ZZ^+\to\ZZ^+\times\ZZ^+$,
$\bs\colon\ZZ^+\times\ZZ^+\to\ZZ^+\times\ZZ^+$, and $\bh\colon\ZZ^+\times\ZZ^+\to\ZZ^+$ defined by $\be(x):=(x,x)$, $\bs(x,y):=(y,x)$, and $\bh(x,y):=x+y$ for all $x,y\in\ZZ^+$. (The maps~$\bs$ and~$\bh$ were already introduced in Section~\ref{S:FirstDiagr}.)

We shall later consider the diagram, denoted by~$\vDu$, obtained by associating order-units to the vertices of the diagram~$\vD$: $1$ is associated to~$\bA$, $(1,1)$ is associated to both~$\bA_j$ and~$\bB_j$ for each $j\in\set{0,1,2}$, and~$2$ is associated to~$\bB$. Observe that~$\vDu$ is a commutative diagram of \emph{pointed monoids}, which means here that $\be(1)=(1,1)=\bs(1,1)$ and $\bh(1,1)=2$.

The diagrams~$\vD$ and~$\vDu$ are represented on the left- and right-hand side of Figure~\ref{Fig:DiagrD}, respectively.

\begin{figure}[htb]
 \[
\xymatrixrowsep{2pc}\xymatrixcolsep{1.5pc}
\def\labelstyle{\displaystyle}
\xymatrix{
& \bB & & & (\bB,2) &\\
\bB_2\ar[ru]^{\bh} & \bB_1\ar[u]_{\bh} &
\bB_0\ar[lu]_{\bh} & (\bB_2,(1,1))\ar[ru]^{\bh} &
(\bB_1,(1,1))\ar[u]_{\bh} & (\bB_0,(1,1))\ar[lu]_{\bh}\\
\bA_0\ar@{=}[u]\ar@{=}[ur] &
\bA_1\ar[ul]_(.7){\bs}\ar@{=}[ur] &
\bA_2\ar@{=}[ul]\ar@{=}[u] & (\bA_0,(1,1))\ar@{=}[u]\ar@{=}[ur] &
(\bA_1,(1,1))\ar[ul]_(.65){\bs}\ar@{=}[ur] &
(\bA_2,(1,1))\ar@{=}[ul]\ar@{=}[u]\\
& \bA\ar[lu]^{\be}\ar[u]_{\be}\ar[ru]_{\be} & & &
(\bA,1)\ar[lu]^{\be}\ar[u]_{\be}\ar[ru]_{\be} &
}
 \]
\caption{The diagrams~$\vD$ and $\vDu$}
\label{Fig:DiagrD}
\end{figure}

The following theorem will involve the functor $\forg\colon\Ring^{\pmeas}\to\CM$ introduced in Definition~\ref{D:TwoFunct}, together with V-semiprimitivity (cf. Definition~\ref{D:Vsemi}).

\begin{theorem}\label{T:NoLift}
There exists no lifting of~$\vD$, with respect to the functor~$\forg$, by any commutative diagram of premeasured rings satisfying the following conditions:
\begin{enumerate}
\item the lifts of~$\bA_0$, $\bA_1$, $\bA_2$ are surjective V-premeasures;

\item the lifts of~$\bB_0$, $\bB_1$, $\bB_2$ are surjective V-measures;

\item the underlying rings of the lifts of~$\bB_0$, $\bB_1$, $\bB_2$ are V-semiprimitive.
\end{enumerate}
\end{theorem}

\begin{proof}
Suppose that we are given a lifting~$\vR$, with respect to the functor~$\forg$, of the diagram~$\vD$, labeled as on Figure~\ref{Fig:RingLift} (\emph{we remind the reader that~$\tilde{f}$ is now dropped from the notation $(f,\tilde{f})$ for morphisms in $\Ring^{\pmeas}$}). We further assume that the conditions (i)--(iii) above hold.

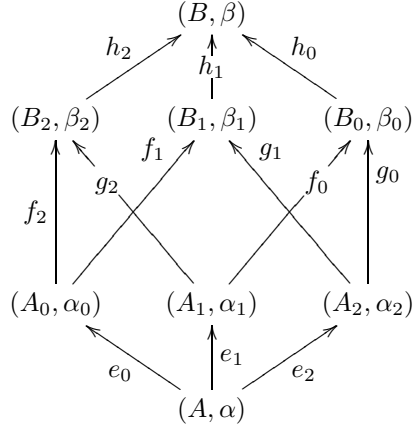
\begin{figure}[htb]
 \[
\xymatrixrowsep{2pc}\xymatrixcolsep{1.5pc}
\def\labelstyle{\displaystyle}
\xymatrix{
& (B,\beta) &\\
(B_2,\beta_2)\ar[ur]^{h_2} & (B_1,\beta_1)\ar[u]|-{h_1} &
(B_0,\beta_0)\ar[lu]_{h_0}\\
& & \\
(A_0,\alpha_0)\ar[uu]^{f_2}\ar[uur]^(.75){f_1} &
(A_1,\alpha_1)\ar[uul]|-(.72){\tvi g_2}\ar[uur]|-(.72){\tvi f_0} &
(A_2,\alpha_2)\ar[uul]_(.75){g_1}\ar[uu]_(.7){g_0}\\
& (A,\alpha)\ar[lu]^{e_0}\ar[u]_{e_1}\ar[ru]_{e_2} &
}
 \]
\caption{A commutative diagram~$\vR$ of premeasured rings lifting~$\vD$}
\label{Fig:RingLift}
\end{figure}

By composing the premeasures in the diagram~$\vR$ with the arrows of the natural equivalence given by the relation $\forg\vR\cong\vD$, we may assume that $\forg\vR=\vD$. This means that the following equalities hold for each $j\in\set{0,1,2}$: $M_{\alpha}=M_{\beta}=\ZZ^+$, $M_{\alpha_j}=M_{\beta_j}=\ZZ^+\times\ZZ^+$, $\tilde{e}_j=\be$, $\tilde{h}_j=\bh$, $\tilde{g}_2=\bs$, and $\tilde{f}_0=\tilde{f}_1=\tilde{f}_2=\tilde{g}_0=\tilde{g}_1=\id_{\ZZ^+\times\ZZ^+}$.

There exists $u\in\Idp_\infty(A)$ such that $\alpha(u)=1$. Set $u_j:=e_j(u)\in\Idp_\infty(A_j)$, for each $j\in\set{0,1,2}$. Then $\alpha_j(u_j)=\tilde{e}_j\alpha(u)=\be(1)=(1,1)$. Likewise, ``propagating'' those new idempotents up the diagram~$\vR$, we obtain idempotent matrices $v_j\in\Idp_\infty(B_j)$, for $j\in\set{0,1,2}$, and $v\in\Idp_\infty(B)$ such that $v_2=f_2(u_0)=g_2(u_1)$ and cyclically, $v=h_2(v_2)=h_1(v_1)=h_0(v_0)$, $\beta_2(v_2)=\beta_1(v_1)=\beta_0(v_0)=(1,1)$, and $\beta(v)=2$.

Define new matrices~$x_j\in\Idp_\infty(A_j)$, for $j\in\set{0,1,2}$, as follows. As $\alpha_j(u_j)=(1,0)+(0,1)$ in~$\bA_j$ and~$\alpha_j$ is a V-premeasure, there exists~$x_j\in\Idp_\infty(A_j)$ such that $x_j\leq u_j$, $\alpha_j(x_j)=(1,0)$, and $\alpha_j(u_j-x_j)=(0,1)$. Set
 \[
 \begin{cases}
 y_0&:=f_0(x_1)\\
 z_0&:=g_0(x_2)
 \end{cases}\,,\quad
 \begin{cases}
 y_1&:=f_1(x_0)\\
 z_1&:=g_1(x_2)
 \end{cases}\,,\quad
 \begin{cases}
 y_2&:=f_2(x_0)\\
 z_2&:=g_2(u_1-x_1)\,.
 \end{cases}
 \]
Then $y_j$ and $z_j$ both belong to $\Idp_\infty(B_j)$ and $y_j,z_j\leq v_j$, for each $j\in\set{0,1,2}$. Further calculations yield easily 
 \begin{multline}\label{Eq:yiapproeqzi}
 \beta_j(y_j)=\beta_j(z_j)=(1,0)\text{ and }
 \beta_j(v_j-y_j)=\beta_j(v_j-z_j)=(0,1)\,,\\
 \text{for each }j\in\set{0,1,2}\,.
 \end{multline}
For example,
 \[
 \beta_2(y_2)=\beta_2(f_2(x_0))=\tilde{f}_2\alpha_0(x_0)
 =\tilde{f}_2(1,0)=(1,0)
 \]
while
 \[
 \beta_2(v_2-z_2)=\beta_2(g_2(x_1))=\tilde{g}_2\alpha_1(x_1)
 =\bs(1,0)=(0,1)\,.
 \]
By evaluating the equation $h_1f_1=h_2f_2$ at~$x_0$, we obtain $h_1(y_1)=h_2(y_2)$. By evaluating $h_0f_0=h_2g_2$ at~$x_1$, we obtain $h_0(y_0)=h_2(v_2-z_2)$. By evaluating $h_0g_0=h_1g_1$ at~$x_2$, we obtain $h_0(z_0)=h_1(z_1)$.

Now, as each~$B_j$ is V-semiprimitive and by~\eqref{Eq:yiapproeqzi}, it follows from Lemma~\ref{L:Bracklift} that $y_j=z_j$ for each $j\in\set{0,1,2}$. Therefore,
 \[
 h_0(y_0)=v-h_2(y_2)=v-h_1(y_1)=v-h_0(y_0)\,,
 \]
a contradiction as~$h_0(y_0)$ and~$v$ are both idempotent with $h_0(y_0)\leq v$ and $v\neq0$.
\end{proof}

We observed in Remark~\ref{Rk:picirctau} that any lifting of a diagram with respect to the functor~$\rV$ yields a lifting of the same diagram with respect to the functor~$\forg$. As the canonical measure on a ring is a V-measure (cf. Example~\ref{Ex:CanVMeas}), we obtain the following immediate consequence of Theorem~\ref{T:NoLift}.

\begin{corollary}\label{C:NoLift1}
There exists no lifting of~$\vD$, with respect to the functor~$\rV$, by any commutative diagram of rings and ring homomorphisms in which the lifts of~$\bB_0$, $\bB_1$ and~$\bB_2$ are all V-semiprimitive.
\end{corollary}

\begin{corollary}\label{C:NoRegLift}
There exists no lifting of~$\vD$, with respect to the functor~$\forg$, by any commutative diagram of V-premeasured regular rings.
\end{corollary}

\begin{proof}
Suppose otherwise, and let~$\vR$ be a commutative diagram of V-premeasured regular rings such that $\forg\vR\cong\vD$. By Lemma~\ref{L:ReflSubcat}, the diagram~$\vR':=\gQr\vR$ is a commutative diagram of \emph{V-measured} regular rings and $\forg\vR'=\forg\vR\cong\vD$. But every regular ring is a semiprimitive exchange ring (cf. Proposition~\ref{P:RegSemipr}), thus it is V-semiprimitive (cf. Proposition~\ref{P:SeExVsem}). This contradicts Theorem~\ref{T:NoLift}.
\end{proof}

A similar proof, using Lemma~\ref{L:ReflSubRR0} instead of Lemma~\ref{L:ReflSubcat}, yields the following.

\begin{corollary}\label{C:NoRR0Lift}
There exists no lifting of~$\vD$, with respect to the functor~$\forg$, by any commutative diagram of V-premeasured C*-algebras of real rank~$0$.
\end{corollary}

By using Remark~\ref{Rk:picirctau}, we obtain the following negative lifting results with respect to the functor~$\rV$.

\begin{corollary}\label{C:NoVLift}
There is no lifting, with respect to the functor~$\rV$, of the diagram~$\vD$ by any diagram of regular rings \pup{resp., C*-algebras of real rank~$0$}.
\end{corollary}

Nonetheless, we shall see that~$\vD$ can be lifted, with respect to the functor~$\rV$, by a diagram of exchange rings (cf. Proposition~\ref{P:cDliftedExch}) and by a diagram of C*-algebras (cf. Proposition~\ref{P:cDliftedC*}).

In order to be able to extend Proposition~\ref{P:NonSplit} to C*-algebras of real rank~$0$, we need to find an analogue of the diagram~$\vD$ in which the arrows are all embeddings. Fortunately, this can be done. We denote by~$\bf,\bg\colon(\ZZ^+)^2\into(\ZZ^+)^4$ and~$\ba$, $\bb$, $\bc\colon(\ZZ^+)^4\into(\ZZ^+)^5$ the maps defined by
 \begin{align*}
 \bf(x,y)&:=(x,x,y,y)\,,\\
 \bg(x,y)&:=(x,y,x,y)\,,\\
 \ba(x_1,x_2,x_4,x_4)&:=(x_1,x_2,x_3,x_4,x_2+x_3)\,,\\
 \bb(x_1,x_2,x_4,x_4)&:=(x_2,x_1,x_3,x_4,x_1+x_4)\,,\\
 \bc(x_1,x_2,x_4,x_4)&:=(x_2,x_3,x_1,x_4,x_1+x_4)\,,
 \end{align*}
for all $x,y,x_1,x_2,x_3,x_4\in\ZZ^+$. Further, we denote by~$\vE$ the commutative diagram of simplicial monoids and normalized monoid order-embeddings represented in Figure~\ref{Fig:DiagrE}.

\begin{figure}[htb]
 \[
\xymatrixrowsep{2pc}\xymatrixcolsep{1.5pc}
\def\labelstyle{\displaystyle}
\xymatrix{
& \bB':=(\ZZ^+)^5 & \\
\tvi\bB'_2:=(\ZZ^+)^4\ar@{_(->}[ru]^{\ba} & \tvi\bB'_1:=(\ZZ^+)^4
\ar@{_(->}[u]_{\bb} &
\tvi\bB'_0:=(\ZZ^+)^4\ar@{^(->}[lu]_{\bc}\\
\bA_0:=(\ZZ^+)^2\ar@{_(->}[u]^{\bf}\ar@{^(->}[ur]_(.75){\bf} &
\bA_1:=(\ZZ^+)^2\ar@{_(->}[ul]_(.6){\bg}\ar@{^(->}[ur]_(.7){\bf} &
\bA_2:=(\ZZ^+)^2\ar@{_(->}[ul]^(.7){\bg}\ar@{^(->}[u]_{\bg}\\
& \bA:=\ZZ^+\ar@{_(->}[lu]^{\be}\ar@{_(->}[u]_{\be}\ar@{^(->}[ru]_{\be} &
}
 \]
\caption{The diagram $\vE$}
\label{Fig:DiagrE}
\end{figure}

\begin{lemma}\label{L:FromEtoD}
There exists a natural transformation $\vec{\chi}\colon\vE\to\vD$ in which all the arrows are surjective pre-\Vhom s.
\end{lemma}

\begin{proof}
We describe the components of the natural transformation.
\begin{itemize}
\item The morphism $\chi_{\bB}\colon\bB'\onto\bB$ is defined by the rule\newline $\chi_{\bB}(x_1,x_2,x_3,x_4,x_5):=x_5$.

\item The morphism $\chi_{\bB_2}\colon\bB'_2\onto\bB_2$ is defined by the rule \newline $\chi_{\bB_2}(x_1,x_2,x_3,x_4):=(x_2,x_3)$.

\item The morphism $\chi_{\bB_1}\colon\bB'_1\onto\bB_1$ is defined by the rule \newline $\chi_{\bB_1}(x_1,x_2,x_3,x_4):=(x_1,x_4)$.

\item The morphism $\chi_{\bB_0}\colon\bB'_0\onto\bB_0$ is defined by the rule \newline $\chi_{\bB_0}(x_1,x_2,x_3,x_4):=(x_1,x_4)$.

\item The morphism $\chi_{\bA_j}\colon\bA_j\onto\bA_j$ is the identity on~$(\ZZ^+)^2$ for each $j<3$, and $\chi_{\bA}\colon\bA\onto\bA$ is the identity on~$\ZZ^+$.

\end{itemize}

Each of these arrows is the canonical projection from~$M\times N$ onto~$M$, for suitable (simplicial) monoids~$M$ and~$N$; hence it is a surjective pre-\Vhom. Verifying that~$\vec{\chi}$ is a natural transformation (which, in reality, amounts to verifying the commutativity of only five squares) is trivial.
\end{proof}

Denote by~$\xSe(1)$ the category of all simplicial monoids with order-unit, with normalized monoid order-embeddings (the latter means that the maps~$f$ satisfy $f(x)\leq f(y)$ implies $x\leq y$, where~$\leq$ denotes the algebraic preordering, which is here an ordering, on the monoid). Actually, all arrows of the diagram~$\vE$ are not only order-embeddings but \emph{coretractions}, that is, they all have a left inverse, so Corollaries~\ref{C:NoLift2} and~\ref{C:NoLift3} extend to the category of all simplicial monoids endowed with those maps as well.

\begin{corollary}\label{C:NoLift2}
There is no functor~$\lift\colon\xSe(1)\to\xRR^{\Vpmeas}$ \pup{resp., \linebreak$\lift\colon\xSe(1)\to\Reg^{\Vpmeas}$} such that $\forg\circ\lift\cong\nobreak\id$.
\end{corollary}

\begin{proof}
We give the proof for $\xRR^{\Vpmeas}$. It suffices to prove that the diagram~$\vE$ has no lifting, with respect to the functor~$\forg$, by a diagram of V-premeasured C*-algebras of real rank~$0$. If~$\vA$ were such a lifting, then, composing all the V-premeasures in~$\vA$ with the corresponding pre-\Vhom s given by Lemma~\ref{L:FromEtoD}, we would obtain a lifting of the diagram~$\vD$, with respect to the functor~$\forg$, by a diagram in~$\xRR^{\Vpmeas}$, a contradiction by Corollary~\ref{C:NoRR0Lift}.

The proof for regular rings is similar, using this time Corollary~\ref{C:NoRegLift}.
\end{proof}

By using Remark~\ref{Rk:picirctau}, we obtain the following extension of Proposition~\ref{P:NonSplit}.

\begin{corollary}\label{C:NoLift3}
There is no functor~$\lift\colon\xSe(1)\to\xRR$ \pup{resp., $\lift\colon\xSe(1)\to\Reg$} such that $\rV\circ\lift\cong\nobreak\id$.
\end{corollary}

We shall see in the subsequent sections that the V-semiprimitivity assumption cannot be dispensed with in the statement of Theorem~\ref{T:NoLift}. We do not know whether Corollary~\ref{C:NoLift2} extends to all C*-algebras, or even to all rings, see Problems~\ref{Pb:RingsFunct} and~\ref{Pb:C*AlgVLift}.

\section{Collapse~$\vD$, collide with exchange rings}\label{S:Collapse}

Five arrows out of six in the middle part of the diagram~$\vD$ (cf. Figure~\ref{Fig:DiagrD}) are identities. This suggests that the structure of~$\vD$ is highly collapsible. In the present section we shall take advantage of this observation and define another diagram, denoted by~$\vL$, obtained by identifying as many ``symmetric'' parts of~$\vD$ as possible while keeping its main unliftability property (Theorem~\ref{T:NoLift}). Denote by~$\vLu$ the diagram obtained by adding the natural order-units to the vertices of~$\vL$. The diagrams~$\vL$ and~$\vLu$ are represented in Figure~\ref{Fig:DiagrL}. The maps~$\be$, $\bh$, and~$\bs$ are already arrows of~$\vD$. 

\begin{figure}[htb]
 \[
\xymatrixrowsep{2pc}\xymatrixcolsep{1.5pc}
\def\labelstyle{\displaystyle}
 \xymatrix{
 \ZZ^+\rloopd{15}{7}_{\id} & &
 (\ZZ^+,2)\rloopd{25}{7}_{\id} & & \\
 \ZZ^+\times\ZZ^+\ar[u]_{\bh} & &
 (\ZZ^+\times\ZZ^+,(1,1))\ar[u]_{\bh} & & \bh(x,y):=x+y\\
 \ZZ^+\times\ZZ^+\rloopd{15}{7}_{\bs}
 \ar@/^.5pc/[u]^{\id}\ar@/_.5pc/[u]_{\id} & &
 (\ZZ^+\times\ZZ^+,(1,1))\rloopd{30}{10}_{\bs}
 \ar@/^.5pc/[u]^{\id}\ar@/_.5pc/[u]_{\id}
 & & \bs(x,y):=(y,x)\\
 \ZZ^+\ar[u]_{\be} & &
 (\ZZ^+,1)\ar[u]_{\be} & & \be(x):=(x,x)
 }
 \]
\caption{The diagrams $\vL$ and $\vLu$}
\label{Fig:DiagrL}
\end{figure}

\begin{proposition}\label{P:NoLiftLu}
There is no lifting~$\vR$, with respect to the functor $\rV$, of the diagram~$\vL$, such that, labeling~$\vR$ as on Figure~\textup{\ref{Fig:LiftL}},
 \begin{equation}\label{Eq:EqnsonvL}
 f\circ e=g\circ e\,,\quad s\circ e=e\,,\quad h\circ f\circ s=v\circ h\circ f\,,
 \quad v\circ h\circ g=h\circ g\,,
 \end{equation}
and~$C$ is V-semiprimitive. In particular, $C$ cannot be a semiprimitive exchange ring, and thus it can be neither a regular ring nor a C*-algebra of real rank~$0$.
\end{proposition}

\begin{figure}[htb]
 \[
\xymatrixrowsep{2pc}\xymatrixcolsep{1.5pc}
\def\labelstyle{\displaystyle}
 \xymatrix{
 D\rloopd{15}{7}_{v} & & \rV(v)=\id_{\ZZ^+}\\
 C\ar[u]_{h} & & \rV(h)=\bh\\
 B\rloopd{15}{7}_{s}
 \ar@/^.5pc/[u]^{f}\ar@/_.5pc/[u]_{g}
 & & \rV(f)=\rV(g)=\id_{\ZZ^+\times\ZZ^+}\,,\ \rV(s)=\bs\\
 A\ar[u]_{e} & & \rV(e)=\be
 }
 \]
\caption{A lifting $\vR$ of $\vL$}
\label{Fig:LiftL}
\end{figure}

\begin{note}
Equations such as $\rV(s)=\bs$, $\rV(f)=\id_{\ZZ^+\times\ZZ^+}$, and so on, are, strictly speaking, stated only up to the natural equivalence between~$\rV(\vR)$ and~$\vLu$.
\end{note}

\begin{proof}
Suppose that~$\vR$ satisfies the given conditions. Then we can ``unfold''~$\vR$ to the commutative diagram represented in Figure~\ref{Fig:CubicLift}, which lifts~$\vD$ with respect to the functor~$\rV$; a contradiction by Corollary~\ref{C:NoLift1}.
\end{proof}

\begin{figure}[htb]
 \[
\xymatrixrowsep{2.5pc}\xymatrixcolsep{1.5pc}
\def\labelstyle{\displaystyle}
\xymatrix{
& D &\\
C\ar[ur]^{h} &C\ar[u]|-{\tvi vh} & C\ar[lu]_{vh}\\
B\ar[u]^{g}\ar[ur]_(.68){\! g} &
B\ar[ul]_(.72){\! fs}\ar[ur]^(.72){f\!\!} &
B\ar[ul]^(.68){g\!}\ar[u]_{g}\\
& A\ar[lu]^{e}\ar[u]_{e}\ar[ru]_{e} &
}
 \]
\caption{A lifting of $\vD$}
\label{Fig:CubicLift}
\end{figure}

\begin{theorem}\label{T:LuLift}
Let~$K$ be a field. Then the diagram~$\vLu$ has a lifting, with respect to the functor $\rVu$, by unital exchange algebras over~$K$, in such a way that the equations~\textup{\eqref{Eq:EqnsonvL}} are satisfied and~$s$ and~$v$ are both involutive automorphisms.
\end{theorem}

\begin{proof}
Denote by $K(\xt)$ (resp., $K[\xt]$) the field of rational functions (resp. the ring of polynomials) over~$K$ in the indeterminate~$\xt$. We denote by~$\sigma$ the unique automorphism of~$K(\xt)$ such that $\sigma(\xt)=1-\xt$. Observe that $\sigma^2=\id_{K(\xt)}$.

The subalgebra
 \begin{equation}\label{Eq:defC0}
 \Ktzero:=\setm{p(\xt)/q(\xt)}{p,q\in K[\xt]\text{ and }q(0)\neq0}
 \end{equation}
is a local subring of~$K(\xt)$, with maximal ideal~$\xt\Ktzero$. In particular, $\Ktzero$ is an exchange ring. The matrices $c_0:=\begin{pmatrix}1&0\\ 0&0\end{pmatrix}$ and $c_1:=1-c_0=\begin{pmatrix}0&0\\ 0&1\end{pmatrix}$ are both idempotent elements in the $K$-algebra
 \begin{equation}\label{Eq:defC}
 C:=\begin{pmatrix}\Ktzero&\xt\Ktzero\\ \xt\Ktzero&\Ktzero
 \end{pmatrix}\,.
 \end{equation}
As $c_0Cc_0\cong c_1Cc_1\cong\Ktzero$ are exchange rings, it follows from Nicholson \cite[Corollary~2.6]{Nicho77} that~$C$ is an exchange subring of $D:=\Mat_2(K(\xt))$. We denote by~$v$ the involutive automorphism of~$D$ defined by the rule
 \[
 v\begin{pmatrix}x_0&x_1\\ x_2&x_3\end{pmatrix}:=
 \begin{pmatrix}\sigma x_3&\sigma x_2\\
 \sigma x_1&\sigma x_0\end{pmatrix}\,,\quad
 \text{for all }x_0,x_1,x_2,x_3\in K(\xt)\,.
 \]
The set $J:=\xt\Mat_2\bigl(\Ktzero\bigr)$ is a two-sided ideal of~$C$. Furthermore, the relation $\det(1-x)\equiv 1\pmod{\xt\Ktzero}$ holds for each $x\in J$; in particular, $1-x$ is invertible in~$C$. Hence~$J$ is contained in the Jacobson radical of~$C$.

Set $B:=K\times K$, $b_0:=(1,0)$, and $b_1:=(0,1)$. Then the map $s\colon B\to B$, $(x,y)\mapsto(y,x)$ is an involutive automorphism of~$B$ with $s(b_0)=b_1$. Furthermore, $\rV(B)\cong\ZZ^+\times\ZZ^+$ with simplicial basis $\set{[b_0]_B,[b_1]_B}$.
We can define a surjective homomorphism $\varphi\colon C\onto B$ of $K$-algebras by the rule
 \[
 \varphi\begin{pmatrix}x_0&x_1\\ x_2&x_3\end{pmatrix}:=
 (x_0(0),x_3(0))\,.
 \]
As the Jacobson radical of~$K\times K$ is zero, it follows that~$J$ contains, and thus is equal to, the Jacobson radical of~$C$. {}From $\rV(J)=\set{0}$ and Proposition~\ref{P:V(R/I)} it follows that $\rV(C)\cong\rV(C/J)$ \emph{via} $[e]_C\mapsto[e+\Mat_\infty(J)]_{C/J}$, thus $\rV(C)\cong\ZZ^+\times\ZZ^+$ with simplicial basis $\set{[c_0]_C,[c_1]_C}$.

Now the matrix
 \begin{equation}\label{Eq:MatrixcExch}
 c:=\begin{pmatrix}(1-\xt)^2(1+2\xt) & \xt(1-\xt)(1+2\xt)\\
 \xt(1-\xt)(3-2\xt) & \xt^2(3-2\xt)\end{pmatrix}
 \end{equation}
is idempotent, and it belongs to~$C$. As $c(0)=\begin{pmatrix}1&0\\ 0&0\end{pmatrix}=c_0(0)$, that is, $c\equiv c_0\pmod{J}$, we get the first key property of~$c$:
 \begin{equation}\label{Eq:cequivc0}
 [c]_C=[c_0]_C\quad\text{and}\quad [1-c]_C=[1-c_0]_C\,.
 \end{equation}
Furthermore, an elementary calculation yields the second key property of~$c$:
 \begin{equation}\label{Eq:v(c)=c}
 v(c)=c\,.
 \end{equation}
Therefore, defining~$h$ as the inclusion map from~$C$ into~$D$ and
defining embeddings of unital $K$-algebras $f\colon K^2\into C$ and $g\colon K^2\into C$ by the rules
 \[
 f(x,y):=xc_0+y(1-c_0)\quad\text{and}\quad
 g(x,y):=xc+y(1-c)\,,\quad\text{for all }x,y\in K\,,
 \]
we obtain the equations $hfs=vhf$ and (using~\eqref{Eq:v(c)=c}) $vhg=hg$.

Finally, set $A:=K$ and $e\colon A\into B$, $x\mapsto(x,x)$. Obviously, $se=e$ and $fe=ge$. We obtain a diagram~$\vR$ as on Figure~\ref{Fig:LiftL}. Verifying that~$\vR$ lifts~$\vLu$, with respect to the abovementioned choices of simplicial bases, amounts to verifying the following:
 \begin{align*}
 \rV(e)([1]_A)&=[b_0]_B+[b_1]_B&&\text{(obvious);}\\
 \rV(f)([b_j]_B)&=[c_j]_C\quad\text{for each }j<2&&\text{(obvious);}\\
 \rV(g)([b_j]_B)&=[c_j]_C\quad\text{for each }j<2&&
 \text{(this follows from~\eqref{Eq:cequivc0})}\\
 \rV(h)([c_j]_C)&=[c_0]_D=[c_1]_D\quad\text{for each }j<2
 &&\text{(obvious);}\\
 \rV(v)([1_D]_D)&=[1_D]_D&&\text{(obvious).} 
 \end{align*}
This concludes the proof.
\end{proof}

The ``unfolding'' technique used in the proof of Proposition~\ref{P:NoLiftLu} yields then immediately the following consequence of Theorem~\ref{T:LuLift}.

\begin{proposition}\label{P:cDliftedExch}
The diagram~$\vDu$ can be lifted, with respect to the functor $\rVu$, by a commutative diagram of unital exchange rings.
\end{proposition}

\section{A lifting of $\vLu$ by C*-algebras}\label{S:C*algLift}

For any topological space~$X$ and any positive integer~$n$, we define the C*-algebra $\AA^{(n)}_X:=\Mat_n(\rC(X,\CC))$. In case~$X$ is given a distinguished point~$0$, we set
 \begin{align*}
 \II^{(n)}_X&:=\setm{x\in\AA^{(n)}_X}{x(0)=0}\,,&&
 \text{an ideal of }\AA^{(n)}_X\,;\\
 \DD^{(n)}_X&:=\begin{pmatrix}\AA^{(n)}_X & \II^{(n)}_X\\
 \II^{(n)}_X & \AA^{(n)}_X\end{pmatrix}\,,&&
 \text{a C*-subalgebra of }\AA^{(2n)}_X;\\
 \JJ^{(n)}&:=\Mat_2(\II^{(n)}_X)\,,&&\text{an ideal of }\DD^{(n)}_X\,.
 \end{align*}
We shall often identify~$\AA^{(n)}_X$ with $\rC(X,\Mat_n(\CC))$ and~$\CC$ with the subalgebra of~$\rC(X,\CC)$ consisting of all constant functions. This way, the C*-algebra
 \[
 \ol{\DD}^{(n)}_X:=\begin{pmatrix}\Mat_n(\CC)&0\\ 0&\Mat_n(\CC)\end{pmatrix}
 \]
is a unital subalgebra of $\DD^{(n)}_X$.
\smallskip

\begin{all}{Convention}
In all the notations $\AA^{(n)}_X$, $\II^{(n)}_X$, $\DD^{(n)}_X$, $\JJ^{(n)}_X$, $\ol{\DD}^{(n)}_X$, defined above, we shall drop the subscript~$X$ and the superscript~$(n)$ in case $X=[0,1]$ \pup{the unit interval of~$\RR$} and $n=1$. We introduced the algebra $\DD=\DD^{(1)}_{[0,1]}$ in Example~\textup{\ref{Ex:semiprnotV}}.
\end{all}

\begin{proposition}\label{P:RRSRAC}
The algebras~$\AA$, $\DD$, and $\Mat_2(\AA)$ have both real rank and stable rank equal to~$1$.
\end{proposition}

\begin{proof}
It follows from Brown and Pedersen~\cite[Proposition~1.1]{BrPe} that~$\AA=\rC([0,1],\CC)$ has real rank~$1$, while, by Rieffel~\cite[Proposition~1.7]{Rief83}, $\AA$ has stable rank~$1$. By Rieffel~\cite[Theorem~3.3]{Rief83}, it follows that $\Mat_2(\AA)$ has also stable rank~$1$. By Brown and Pedersen~\cite[Proposition~1.2]{BrPe}, it follows that~$\Mat_2(\AA)$ has real rank~$1$ as well.

As~$\AA$ has stable rank~$1$, the argument of the second part of the proof of~\cite[Theorem~3.3]{Rief83} shows easily that~$\DD$ has stable rank~$1$. By~\cite[Proposition~1.2]{BrPe}, it follows that the real rank of~$\DD$ is either~$0$ or~$1$. As~$\AA$ is isomorphic to a hereditary subalgebra of~$\DD$, it follows from \cite[Corollary~2.8]{BrPe} that~$\DD$ has real rank~$1$.
\end{proof}

For each idempotent $e\in\DD^{(n)}_X$, the trace $\tr(e(x))$ is a nonnegative integer for each $x\in X$. As the map $\tr\circ\,e$ is continuous on~$X$, it follows that if~$X$ is connected (\emph{which will always be the case throughout this section}), then the value of~$\tr(e(x))$, for $x\in X$, is constant. We shall denote this value by~$\tr(e)$ and call it the trace of~$e$. Observe that the trace of~$e$ is equal to the rank of~$e(x)$, for each $x\in X$.

The following result is contained in Examples~1.4.2 and~5.1.3(c) of Blackadar~\cite{Black98}.

\begin{lemma}\label{L:VC[01]}
Let $a,b\in\RR$ with $a\leq b$. Then $\rV(\rC([a,b],\CC))\cong\ZZ^+$. Furthermore, for any positive integer~$n$, two idempotent elements of $\Mat_n(\rC([a,b],\CC))$ are equivalent if{f} they have the same trace.
\end{lemma}

\begin{lemma}\label{L:eequivrhonXe}
Let $e\in\DD^{(n)}_{[0,1]}$ be a projection. Then~$e$ and~$e(0)$ are unitarily equivalent in $\DD^{(n)}_{[0,1]}$.
\end{lemma}

\begin{proof}
Set $\ol{e}:=e(0)$. The common value $r:=\tr(e)=\tr(\ol{e})$ is an integer with $0\leq r\leq 2n$. As~$e$ is continuous at~$0$, there is $\eps\in(0,1]$ such that $\norm{e(t)-e(0)}\leq 1/2$ for each $t\in[0,\eps]$. Thus we obtain the inequality (between elements of $\DD^{(n)}_{[0,\eps]}$)
 \[
 \norm{e\res_{[0,\eps]}-\ol{e}\res_{[0,\eps]}}\leq 1/2\,.
 \]
As $e\res_{[0,\eps]}$ and $\ol{e}\res_{[0,\eps]}$ are both projections of $\DD^{(n)}_{[0,\eps]}$, they are unitarily equivalent in $\DD^{(n)}_{[0,\eps]}$, that is, there exists a unitary $u\in\DD^{(n)}_{[0,\eps]}$ such that
 \[
 e\res_{[0,\eps]}=u\ol{e}\res_{[0,\eps]}u^*\,.
 \]
Extend~$u$ to a unitary element of $\DD^{(n)}_{[0,1]}$ by setting $u(t):=u(\eps)$ for each $t\in[\eps,1]$. By replacing~$e$ by $u^*eu$, we may thus assume the following:
 \begin{equation}\label{Eq:ecstat0}
 e(t)=e(0)\quad\text{for each }t\in[0,\eps]\,. 
 \end{equation}
As $e\res_{[\eps,1]}$ and $\ol{e}\res_{[\eps,1]}$ are projections of $\Mat_{2n}(\rC([\eps,1],\CC))$ with the same rank~$r$, it follows from Lemma~\ref{L:VC[01]} that they are unitarily equivalent in $\Mat_{2n}(\rC([\eps,1],\CC))$, so there exists a unitary $v\in\Mat_{2n}(\rC([\eps,1],\CC))$ such that
 \begin{equation}\label{Eq:efromebar}
 e\res_{[\eps,1]}=v\ol{e}\res_{[\eps,1]}v^*\,.
 \end{equation}
Extend~$v$ to a unitary element of $\DD^{(n)}_{[0,1]}$ by setting $v(t):=v(\eps)$ for each $t\in[0,\eps]$. By putting~\eqref{Eq:ecstat0} and~\eqref{Eq:efromebar} together, we obtain easily that $e=v\ol{e}v^*$.
\end{proof}

As in Section~\ref{S:Collapse}, we set $c_0:=\begin{pmatrix}1&0\\ 0&0\end{pmatrix}$ and $c_1:=1-c_0=\begin{pmatrix}0&0\\ 0&1\end{pmatrix}$. By using the canonical isomorphism $\Mat_n(\DD)\cong\DD^{(n)}_{[0,1]}$ together with Lemma~\ref{L:eequivrhonXe}, we obtain the following description of the nonstable $\rK_0$-theory of~$\DD$.

\begin{corollary}\label{C:V(B1[01]}
Denote by $\rho\colon C\onto\ol{\DD}$, $x\mapsto x(0)$ the canonical retraction. Then $\rV(\rho)\colon\rV(\DD)\to\rV(\ol{\DD})\cong\ZZ^+\times\ZZ^+$ is an isomorphism, and $\set{[c_0]_C,[c_1]_C}$ is a simplicial basis of~$\rV(\DD)$.
\end{corollary}

In particular, idempotent matrices $a,b\in\Idp_\infty(\DD)$ are equivalent in $\Mat_\infty(\DD)$ if{f} $a+\Mat_\infty(\JJ)$ and $b+\Mat_\infty(\JJ)$ are equivalent in $\Mat_\infty(\DD/\JJ)$. Observe that $\DD/\JJ\cong\CC^2$.

\begin{theorem}\label{T:LuLiftC*}
The diagram~$\vLu$ has a lifting, with respect to the functor $\rVu$, by unital C*-subalgebras of~$\Mat_2(\rC([0,1],\CC))$, in such a way that the equations~\textup{\eqref{Eq:EqnsonvL}} are satisfied and~$s$ and~$v$ are both involutive automorphisms.
\end{theorem}

\begin{proof}
The argument is similar to the one of the proof of Theorem~\ref{T:LuLift}. The role played by $K(\xt)$ in that proof is taken up here by~$\AA:=\rC([0,1])$. We denote by~$\sigma$ the automorphism of~$\AA$ defined by
 \[
 \sigma(x)(t):=x(1-t)\,,\quad\text{for each }x\in A\text{ and each }t\in[0,1]\,.
 \]
We set $C:=\DD$, $D:=\Mat_2(\AA)$, and we denote by~$v$ the involutive automorphism of~$\Mat_2(\AA)$ defined by the rule
 \[
 v\begin{pmatrix}x_0&x_1\\ x_2&x_3\end{pmatrix}:=
 \begin{pmatrix}\sigma x_3&\sigma x_2\\
 \sigma x_1&\sigma x_0\end{pmatrix}\,,\quad
 \text{for all }x_0,x_1,x_2,x_3\in A\,.
 \]
Set $B:=\CC\times\CC$, $b_0:=(1,0)$, and $b_1:=(0,1)$. Then the map $s\colon B\to B$, $(x,y)\mapsto(y,x)$ is an involutive C*-automorphism of~$B$, which switches~$b_0$ and~$b_1$. Furthermore, $\rV(B)\cong\ZZ^+\times\ZZ^+$ with simplicial basis $\set{[b_0]_B,[b_1]_B}$.
Denote by $z\colon[0,1]\into\CC$ the inclusion map. The matrix
 \[
 c:=\begin{pmatrix}1-z & \sqrt{z(1-z)}\\
 \sqrt{z(1-z)} & z\end{pmatrix}
 \]
is a projection of~$\DD$. As $c(0)=\begin{pmatrix}1&0\\ 0&0\end{pmatrix}=c_0(0)$, that is, $c\equiv c_0\pmod{\JJ}$, it follows from Corollary~\ref{C:V(B1[01]} that
 \[
 [c]_C=[c_0]_C\quad\text{and}\quad [1-c]_C=[1-c_0]_C\,.
 \]
Furthermore, an elementary calculation yields the second key property of~$c$, namely
 \begin{equation}\label{Eq:v(c)=c2}
 v(c)=c\,.
 \end{equation}
Therefore, defining~$h$ as the inclusion map from~$\DD$ into~$\Mat_2(\AA)$ and defining \Cemb s $f\colon\CC^2\into\DD$ and $g\colon\CC^2\into\DD$ by the rules
 \[
 f(x,y):=xc_0+y(1-c_0)\quad\text{and}\quad
 g(x,y):=xc+y(1-c)\,,\quad\text{for all }x,y\in K\,,
 \]
we obtain the equations $hfs=vhf$ and (using~\eqref{Eq:v(c)=c2}) $vhg=hg$.

Finally, let $e\colon\CC\into B$, $x\mapsto(x,x)$. Obviously, $se=e$ and $fe=ge$. We obtain a diagram~$\vR$ as on Figure~\ref{Fig:LiftL}, with~$\CC$ in place of~$A$. Verifying that~$\vR$ lifts~$\vLu$, with respect to the abovementioned choices of simplicial bases, is done in a similar fashion as at the end of the proof of Theorem~\ref{T:LuLift}.
\end{proof}

The ``unfolding'' technique used in the proof of Proposition~\ref{P:NoLiftLu} yields then immediately the following consequence of Theorem~\ref{T:LuLiftC*}.

\begin{proposition}\label{P:cDliftedC*}
The diagram~$\vDu$ can be lifted, with respect to the functor $\rVu$, by a commutative diagram of unital C*-algebras and unital C*-homomorphisms.
\end{proposition}

The algebra~$\DD$ can also be applied to find a lifting of the diagram~$\vK$ represented in Figure~\ref{Fig:OneLoop}.

\begin{proposition}\label{P:vecCLift}
The diagram $\vK$ can be lifted, with respect to the functor~$\rV$, by a diagram of unital C*-algebras and unital C*-homomorphisms such that, labeling the lifting as on Figure~\textup{\ref{Fig:OneLoop}}, $h\circ s=h$.
\end{proposition}

\begin{proof}
We shall define unital C*-algebras~$R$ and~$S$ with unital C*-homomorphisms $s\colon R\to R$ and $h\colon R\to S$ such that $h\circ s=h$ and, denoting by~$\vR$ the diagram represented in the right hand side of Figure~\ref{Fig:OneLoop}, $\rV\vR\cong\vK$.

Set $R:=\DD$ and $S:=\Mat_2(\CC)$, with the already introduced~$c_0$ and~$c_1$. We set
 \[
 \gamma(t):=\begin{pmatrix}\sin((\pi/2)t)&\cos((\pi/2)t)\\
 -\cos((\pi/2)t)&\sin((\pi/2)t)\end{pmatrix}\,,
 \quad\text{for each }t\in[0,1]\,.
 \]
Observe that $\gamma$ is a unitary element of $\Mat_2(\AA)$ and $\gamma\notin\DD$. Furthermore, set $s(x):=\gamma\cdot x\cdot\gamma^*$, for each $x\in\DD$. Then
 \[
 s(c_0)(0)=\begin{pmatrix}0&1\\ -1&0\end{pmatrix}
 \begin{pmatrix}1&0\\ 0&0\end{pmatrix}
 \begin{pmatrix}0&-1\\ 1&0\end{pmatrix}=c_1\,,
 \]
and, similarly, $s(c_1)(0)=c_0$. It follows that~$s$ is an automorphism of~$\DD$, and, by Lemma~\ref{L:eequivrhonXe}, $s(c_0)\sim c_1$ and $s(c_1)\sim c_0$ in~$\DD$. Hence, $\rV(s)$ switches $[c_0]_\DD$ and $[c_1]_\DD$. Observe that by Corollary~\ref{C:V(B1[01]}, $\rV(\DD)$ is simplicial with simplicial basis $\set{[c_0]_C,[c_1]_C}$.

Now let $h\colon C\to\Mat_2(\CC)$, $x\mapsto x(1)$. Then $h(c_j)=c_j$ for each $j\in\set{0,1}$. Furthermore, from $\gamma(1)=1$ it follows that for each $x\in\DD$,
 \[
 (h\circ s)(x)=h(\gamma\cdot x\cdot\gamma^*)=
 \gamma(1)\cdot x(1)\cdot\gamma(1)^*=x(1)=h(x)\,,
 \]
so $h\circ s=h$, and so~$s$ and~$h$ are as desired.
\end{proof}

We do not know whether it is possible to ensure simultaneously $h\circ s=h$ and $s^2=\id_R$ in a lifting of~$\vK$, with respect to the functor~$\rV$, by unital C*-algebras.

\section{Back to the transfinite: an exchange ring of cardinality $\aleph_3$}\label{S:Al3}

Our first non-lifting result, Proposition~\ref{P:NonSplit}, started with a counterexample of cardinality~$\aleph_2$. We applied the combinatorial core of that result in Theorem~\ref{T:NoLift}. Now the cycle swings back on itself: we shall get further negative representation results, algebraically stronger than those that we used in order to get Proposition~\ref{P:NonSplit}, now in cardinality~$\aleph_3$ (the present form of Theorem~\ref{T:NoLift} would not yield~$\aleph_2$).

Throughout this section we fix a field~$K$, and we denote by~$\Ktzero$ (cf.~\eqref{Eq:defC0}) and $C_K:=\begin{pmatrix}\Ktzero&\xt\Ktzero\\ \xt\Ktzero&\Ktzero\end{pmatrix}$ (cf.~\eqref{Eq:defC}) the $K$-algebras introduced in the proof of Theorem~\ref{T:LuLift}. We also denote by~$\xA_K$ the closure, under finite direct products and direct limits (of unital $K$-algebras), of the set
 \begin{equation}\label{Eq:DefncR}
 \cR_K:=\set{K}\cup\set{C_K}\cup\set{\Mat_2(K(\xt))}\,.
 \end{equation}
As each member of~$\cR_K$ is a unital exchange $K$-algebra with index of nilpotence at most~$2$, and as these properties are preserved under finite products and direct limits, every member~$R$ of~$\xA_K$ is a unital exchange $K$-algebra with index of nilpotence at most~$2$. As, by Proposition~\ref{P:V(Exch)}, $\rV(R)$ is a (conical) refinement monoid, it follows from Lemma~\ref{L:BasicVcAK} that $\rV(R)$ is the positive cone of a dimension group in which~$[1_R]$ has index at most two. In particular, by Yu~\cite[Theorem~9]{Yu95}, $R$ has stable rank~$1$.

Now recall that~$\CM$ denotes the category of all conical \cm s with monoid homomorphisms (cf. Definition~\ref{D:PointedMon}). We shall also define the subcategory $\CM^\Rightarrow$ with the same objects but where the morphisms are the pre-\Vhom s (cf. Definition~\ref{D:Vhom}), and the full subcategory~$\CM^{\les\aleph_0}$ of all \emph{countable} (i.e., at most countable) conical \cm s.

Denote by~$\Phi\colon\xA_K\to\CM$ the functor~$\rV$ (i.e., $R\mapsto\rV(R)$, $f\mapsto\rV(f)$).

We introduce in Definition~3.8.1 and Definition~3.8.2 of Gillibert and Wehrung~\cite{GiWe2} attributes of certain collections of categories and functors, called \emph{left larderhood} and \emph{right larderhood}. We shall use the notation used in these two definitions, yet trying to spell out in detail the verifications to be performed, and we shall start with the (traditionally easier) left part.

\begin{lemma}\label{L:cAKLeftLard}
The quadruple $(\xA_K,\CM,\CM^\Rightarrow,\Phi)$ is a left larder.
\end{lemma}

\begin{proof}
We check one after another the items defining left larderhood.

\begin{itemize}
\item (CLOS$(\xA_K)$) and (PROD$(\xA_K)$): by definition, $\xA_K$ is closed under direct limits and finite products.

\item (CONT$(\Phi)$): it is well-known that the functor~$\rV$ preserves direct limits.

\item (PROJ$(\Phi,\CM^\Rightarrow)$): we must check that $\rV(f)$ is a double arrow (i.e., a pre-\Vhom), for each morphism $f\colon R\to S$ in~$\xA_K$ which is a direct limit (in the category of all arrows of~$\xA_K$) $f=\varinjlim_{j\in J}f_j$ of projections $f_j\colon R_j\times S_j\onto R_j$, $(x,y)\mapsto x$, for each $j\in J$ ($J$ is a directed partially ordered set). As $\rV(R_j\times S_j)\cong\rV(R_j)\times\rV(S_j)$, with $\rV(f_j)$ the canonical projection $\rV(R_j)\times\rV(S_j)\onto\rV(R_j)$, $\rV(f_j)$ is obviously a pre-\Vhom. As the class of pre-\Vhom s is easily seen to be closed under direct limits (within the category of all arrows of~$\CM$) and the functor~$\rV$ preserves direct limits, it follows that~$\rV(f)$ is a pre-\Vhom.\qed
\end{itemize}
\renewcommand{\qed}{}
\end{proof}

In order to fill the ``right larder'' part, we need, for C*-algebras of real rank~$0$,

\begin{itemize}
\item $\xRR$, the category of all C*-algebras of real rank~$0$ with \Chom s (cf. Definition~\ref{D:CatRings1}).

\item $\xRRs$, the category of all \emph{separable} members of~$\xRR$.

\item $\gYa$ denotes the functor~$\rV$ (i.e., $B\mapsto\rV(B)$, $f\mapsto\rV(f)$) from~$\xRR$ to~$\CM$.
\end{itemize}

For the side of regular rings, we need

\begin{itemize}
\item $\Reg$, the category of all regular rings with ring homomorphisms.

\item $\Reg^{\les\aleph_0}$, the category of all \emph{countable} regular rings.

\item $\gYr$ denotes the functor~$\rV$ (i.e., $B\mapsto\rV(B)$, $f\mapsto\rV(f)$) from~$\Reg$ to~$\CM$.
\end{itemize}

Back to C*-algebras of real rank~$0$, we prove

\begin{lemma}\label{L:cBRightLard}
The $6$-uple $(\xRR,\xRRs,\CM,\CM^{\les\aleph_0},\CM^\Rightarrow,\gYa)$ is a right $\aleph_1$-larder.
\end{lemma}

\begin{proof}
The part (PRES${}_{\aleph_1}(\CM^{\les\aleph_0}, \gYa)$) involves the notion of a \emph{weakly $\aleph_1$-pre\-sent\-ed structure} introduced in Gillibert and Wehrung \cite[Definition~1.3.2]{GiWe2}. It is, basically, easy: every $B\in\xRRs$ is separable, thus (cf. the Note on page~28 in Blackadar~\cite{Black98}) $\rV(B)$ is countable, thus (cf. Gillibert and Wehrung \cite[Proposition~4.2.3]{GiWe2}) weakly $\aleph_1$-presented.

The part (LS${}_{\aleph_1}^{\mathrm{r}}(B)$), for a given $B\in\xRR$, is less obvious. We are given a countable conical refinement monoid~$M$ and a pre-\Vhom\ $\psi\colon\rV(B)\Rightarrow\nobreak M$, together with a sequence $f_n\colon B_n\to B$ of unital \Chom s (the~$f_n$ can be assumed to be monomorphisms, but this will play no role here). We are trying to find a separable, unital C*-subalgebra~$C$ of~$B$, of real rank~$0$, containing $\bigcup_{n\in\ZZ^+}f_n(B_n)$, such that, if $e_C\colon C\into B$ denotes the inclusion map, then $\psi\circ\rV(e_C)$ is a surjective pre-\Vhom.

For each $n\in\ZZ^+$, denote by~$r_n\colon\RR\to\RR$ the continuous function defined by $r_n(x)=0$ if $x\leq\frac{1}{n+3}$, $r_n(x)=1$ if $x\geq\frac{n+2}{n+3}$, and~$r_n$ is affine on the interval $\bigl[\frac{1}{n+3},\frac{n+2}{n+3}\bigr]$. We say that a countable subset~$X$ of~$B$ is \emph{FFC-closed} (``FFC'' stands for ``Flat Function Calculus'') if $r_n(x)\in\Mat_\infty(X)$ for each self-adjoint element $x\in\Mat_\infty(X)$ and each $n\in\ZZ^+$. Furthermore, denote by~$\cF$ the set of all countable, FFC-closed, *-closed unital $\QQ[i]$-subalgebras of~$B$. A classical argument about approximating projections on a dense set (see, for example, Blackadar \cite[Proposition~4.5.1]{Black98}) shows that every member~$X$ of~$\cF$ satisfies the following:
 \begin{multline}\label{Eq:DenseProj}
 \text{For each }\eps>0\text{ and each projection }
 e\in\Mat_\infty(\ol{X})\,,\\
 \text{ there exists a projection }x\in\Mat_\infty(X)\text{ such that }
 \norm{e-x}<\eps\,. 
 \end{multline}

\begin{sclaim}
For all $c\in\Idp_\infty(B)$ and all $\alpha,\beta\in M$ with $\psi([c]_B)=\alpha+\beta$, there are orthogonal idempotents $a,b\in\Idp_\infty(B)$ such that $c=a+b$ while $\psi([a]_B)=\alpha$ and $\psi([b]_B)=\beta$.
\end{sclaim}

\begin{scproof}
As~$\psi$ is a pre-\Vhom, there are $\alpha',\beta'\in\rV(B)$ such that $[c]_B=\alpha'+\beta'$ while $\psi(\alpha')=\alpha$ and $\psi(\beta')=\beta$. By Lemma~\ref{L:VmeasDec}, there are orthogonal idempotents $a,b\in\Idp_\infty(B)$ such that $c=a+b$ while $[a]_B=\alpha'$ and $[b]_B=\beta'$. Hence~$a$ and~$b$ are as required.
\end{scproof}

For each $n\in\ZZ^+$, pick a countable dense subset $X_n\subseteq B_n$.
As~$\psi$ is surjective and~$M$ is countable, there exists $Y_0\in\cF$ containing $\bigcup_{n\in\ZZ^+}f_n(X_n)$ such that $M=\setm{\psi([e]_B)}{e\in\Idp_\infty(Y_0)}$.

Suppose that~$Y_j\in\cF$ has been constructed for each $j\leq 2k$. As~$B$ has real rank~$0$ and~$Y_{2k}$ is countable, there exists $Y_{2k+1}\in\cF$ containing~$Y_{2k}$ such that every self-adjoint element of $\QQ[i]\times Y_{2k}$, viewed as a subset of the unitization of~$B$, lies within $1/(k+1)$ of some invertible self-adjoint element of the unitization of~$Y_{2k+1}$.

As~$Y_{2k+1}$ and~$M$ are both countable and by the Claim above, there is $Y_{2k+2}\in\cF$ containing~$Y_{2k+1}$ such that for all $c\in\Idp_\infty(Y_{2k+1})$ and $\alpha,\beta\in M$ with $\psi([c]_B)=\alpha+\beta$, there are orthogonal idempotents $a,b\in\Idp_\infty(Y_{2k+2})$ such that $c=a+b$ while $\psi([a]_B)=\alpha$ and $\psi([b]_B)=\beta$.

By construction, the union $Y:=\bigcup_{k\in\ZZ^+}Y_k$ belongs to~$\cF$ and the closure~$\ol{Y}$ has real rank~$0$. Furthermore, it follows from~\eqref{Eq:DenseProj} that every projection of $\Mat_\infty(\ol{Y})$ is equivalent to some projection in $\Mat_\infty(Y)$; hence, by the construction of~$Y_{2k+2}$ from~$Y_{2k+1}$, it follows that $\psi\circ\rV(e_{\ol{Y}})$ is a pre-\Vhom. Furthermore, from $\bigcup_{n\in\ZZ^+}f_n(X_n)\subseteq Y_0\subseteq Y$ it follows that $\bigcup_{n\in\ZZ^+}f_n(B_n)\subseteq\ol{Y}$. Therefore, $C:=\ol{Y}$ is as required.
\end{proof}

Consider again the set~$\cR_K$ of exchange algebras defined in~\eqref{Eq:DefncR}. As $\Phi(\cR_K)$ is contained in $\CM^{\les\aleph_0}$, we obtain the following consequence of Lemmas~\ref{L:cAKLeftLard} and~\ref{L:cBRightLard} together with the (obvious) \cite[Proposition~3.8.3]{GiWe2}.

\begin{lemma}\label{L:AKRR0Larder}
The $8$-uple $(\xA_K,\xRR,\CM,\cR_K,\xRRs,\CM^\Rightarrow,\Phi, \gYa)$ is an $\aleph_1$-larder.
\end{lemma}

The analogue of Lemma~\ref{L:cBRightLard} for regular rings is the following.

\begin{lemma}\label{L:cBRightLardReg}
The $6$-uple $(\Reg,\Reg^{\les\aleph_0},\CM,\CM^{\les\aleph_0},\CM^\Rightarrow,\gYr)$ is a right $\aleph_1$-larder.
\end{lemma}

The proof of Lemma~\ref{L:cBRightLardReg} is similar, and actually slightly easier (because there is no topology involved) than the one of Lemma~\ref{L:cBRightLard}, thus we omit it.

Lemma~\ref{L:AKRR0Larder} is modified in a similar fashion.

\begin{lemma}\label{L:AKRegLarder}
The $8$-uple $(\xA_K,\Reg,\CM,\cR_K,\Reg^{\les\aleph_0},\CM^\Rightarrow,\Phi,\gYr)$ is an $\aleph_1$-larder.
\end{lemma}

Now everything is ready for the proof of the following result.

\begin{theorem}\label{T:CXAKRR0}
For any field~$K$, there exists a unital exchange $K$-algebra~$R_K$ that satisfies the following properties:
\begin{enumerate}
\item $R_K$ has index of nilpotence~$2$ \pup{thus stable rank~$1$}.

\item For every C*-algebra of real rank~$0$ \pup{resp., for every regular ring}~$B$, there is no surjective pre-\Vhom\ from~$\rV(B)$ onto~$\rV(R_K)$.
\end{enumerate}
In particular, there is no C*-algebra of real rank~$0$ \pup{resp., no regular ring}~$B$ such that $\rV(R_K)\cong\rV(B)$. Furthermore, $R_K$ can be constructed as the \pup{unital} direct limit of a system of~$\aleph_3$ finite products of members of~$\cR_K$. In particular, if~$K$ is countable \pup{or, more generally, has at most~$\aleph_3$ elements}, then $\card R_K\leq\aleph_3$.
\end{theorem}

\begin{proof}
We apply \cite[Lemma 3.4.2]{GiWe2}, called there CLL, to the following data:
\begin{itemize}
\item $\lambda=\mu=\aleph_1$;

\item $P$ is the powerset lattice of the three-element set $\set{0,1,2}$;

\item $\vA$ is the commutative diagram of unital exchange $K$-algebras, given in Proposition~\ref{P:cDliftedExch}, such that~$\rVu\vA\cong\vDu$. Observe that all objects of~$\vA$ belong to~$\cR_K$ (cf. \eqref{Eq:DefncR});

\item $\Lambda$ is one of the $\aleph_1$-larders given by either Lemma~\ref{L:AKRR0Larder} or Lemma~\ref{L:AKRegLarder}. The structure $\xF(X)\otimes\vA$ involved in the statement of CLL, which will turn out to be the desired counterexample~$R_K$, is the same for both larders (it depends only of~$\xA_K$, $\vA$, $P$, and~$X$).
\end{itemize}

In order for the assumptions underlying CLL to be fulfilled, we need~$P$ to admit an $\aleph_1$-lifter (cf. \cite[Definition 3.2.1]{GiWe2}) $(X,\bX)$ such that $\card X\leq\aleph_3$. As~$P$ is a finite lattice, it is a so-called \emph{almost join-semilattice} (cf. \cite[Definition~2.1.2]{GiWe2}), thus, according to \cite[Corollary 3.5.8]{GiWe2}, it is sufficient to prove that the relation $(\aleph_3,{<}\aleph_1)\leadsto P$ (cf. \cite[Definition~3.1]{GiWe1}, also \cite[Definition 3.5.1]{GiWe2}) holds. According to the definition of the Kuratowski index $\kur(P)$ given in \cite[Definition~4.1]{GiWe1}, it suffices to prove that $\kur(P)\leq 3$. As~$P$ has exactly three \jirr\ elements, this is a trivial consequence of \cite[Proposition~4.2]{GiWe1}.

The statement of CLL involves a ``$P$-scaled Boolean algebra'', denoted there by~$\xF(X)$ (where $(X,\bX)$ is the abovementioned $\aleph_1$-lifter), and the ``condensate'' $\xF(X)\otimes\vA$. A $P$-scaled Boolean algebra is a Boolean algebra endowed with a collection, indexed by~$P$, of ideals, subjected to certain constraints.
For our present purposes, neither the exact definition of a $P$-scaled Boolean algebra (cf. \cite[Definition 2.2.3]{GiWe2}), nor the exact descriptions of the constructions of~$\xF(X)$ (cf. \cite[Lemma~2.6.5]{GiWe2}) and $\xF(X)\otimes\vA$ (cf. \cite[Section 3.1]{GiWe2}), will matter.

We set $R_K:=\xF(X)\otimes\vA$. The statement that~$R_K$ is a direct limit of finite products of members of~$\cR_K$, together with the cardinality bound on~$R_K$, are immediate consequences of the following facts:
\begin{itemize}
\item $\xF(X)$ is a $P$-scaled Boolean algebra with at most~$\aleph_3$ elements.

\item Thus $\xF(X)=\varinjlim_{j\in I}\bB_j$, for a directed poset~$I$ of cardinality $\aleph_3$ and ``compact'' (i.e., here, finitely presented) $P$-scaled Boolean algebras~$\bB_j$ (cf. \cite[Proposition~2.4.6]{GiWe2}).

\item $\xF(X)\otimes\vA=
\varinjlim_{j\in I}(\bB_j\otimes\vA)$. This follows from \cite[Proposition 3.1.4]{GiWe2}.

\item Each $\bB_j\otimes\vA$ is a finite product of members of~$\vA$. This follows from \cite[Definition 3.1.1]{GiWe2}.
\end{itemize}

Suppose that there exists a surjective pre-\Vhom\ $\chi\colon\rV(B)\Rightarrow\rV(R_K)$, for a C*-algebra~$B$ with real rank~$0$. As all the assumptions underlying CLL are satisfied, there are a $P$-indexed diagram~$\vec B$ of C*-algebras of real rank~$0$ and a natural transformation $\vec{\chi}\colon\rV\vec B\Rightarrow\rV\vA$. As $\rV\vA\cong\vD$ (cf. Proposition~\ref{P:cDliftedExch}), we get a natural transformation, that we shall denote again by~$\vec{\chi}$, now $\vec{\chi}\colon\rV\vec B\Rightarrow\vD$. The double arrow notation for the natural transformation~$\vec{\chi}$ means here that each component $\chi_p\colon\rV(B_p)\to D_p$ (for $p\in P$) of~$\vec{\chi}$ is a surjective pre-\Vhom.

It follows that the rule $\mu_p(e):=\chi_p([e]_{B_p})$, for $e\in\Idp_\infty(B_p)$, defines a $D_p$-valued surjective V-measure on~$B_p$. We thus obtain a diagram~$\vec{C}$ of surjective V-premeasures on (unital) C*-algebras of real rank~$0$ such that $\forg\vec{C}\cong\vD$. But this contradicts Corollary~\ref{C:NoRR0Lift}.

This completes the proof that~$R_K$ is the desired counterexample for C*-algebras of real rank~$0$.

We also need to prove that there exists no surjective pre-\Vhom\ $\chi\colon\rV(B)\Rightarrow\rV(R_K)$, for any regular ring~$B$. The proof is similar to the one for C*-algebras of real rank~$0$. We need to use Lemmas~\ref{L:cBRightLardReg} and~\ref{L:AKRegLarder} instead of Lemmas~\ref{L:cBRightLard} and~\ref{L:AKRR0Larder}, and Corollary~\ref{C:NoVLift} instead of Corollary~\ref{C:NoRR0Lift}.
\end{proof}

\section{An $\aleph_3$-separable C*-algebra}\label{S:Al3Sep}

For an infinite cardinal~$\kappa$, we say that a C*-algebra is \emph{$\kappa$-separable} if it has a dense subset of cardinality at most~$\kappa$. In particular, $\aleph_0$-separable means separable. We claim that what we did in Section~\ref{S:Al3}, using the diagram of exchange rings constructed in Section~\ref{S:Collapse}, can be done for C*-algebras, using the diagram of C*-algebras constructed in Section~\ref{S:C*algLift}.

Consider again the C*-algebra~$\DD$ used in Example~\ref{Ex:semiprnotV} and Section~\ref{S:C*algLift}, and set
 \begin{equation}\label{Eq:DefncA}
 \cA:=\set{\CC}\cup\set{\DD}\cup\set{\Mat_2(\rC([0,1),\CC))}\,.
 \end{equation}
The relevant analogue of Theorem~\ref{T:CXAKRR0} is the following.

\begin{theorem}\label{T:C*CXAKRR0}
There exists an $\aleph_3$-separable unital C*-algebra~$E$ that satisfies the following properties.
\begin{enumerate}
\item $E$ has index of nilpotence~$2$, and real rank and stable rank both equal to~$1$.

\item $(\rV(E),[1_E])$ is the positive cone of a dimension group with order-unit of index two.

\item For every C*-algebra of real rank~$0$ \pup{resp., every regular ring}~$B$, there exists no surjective pre-\Vhom\ from~$\rV(B)$ onto~$\rV(E)$.

\end{enumerate}
In particular, there is no C*-algebra of real rank~$0$ \pup{resp., no regular ring}~$B$ such that $\rV(E)\cong\rV(B)$. Furthermore, $E$ can be constructed as the \pup{unital} C*-direct limit of a system of~$\aleph_3$ finite products of members of~$\cA$.
\end{theorem}

While stating that the nonstable $\rK_0$-theory of C*-algebras is contained neither in the one of regular rings nor in the one of C*-algebras of real rank~$0$ is no big deal (for~$\rV(A)$ may not have refinement, even for a unital C*-algebra~$A$, cf. Blackadar \cite[Example~5.1.3(e)]{Black98}), Theorem~\ref{T:C*CXAKRR0} states this fact even for those C*-algebras whose nonstable $\rK_0$-theory is a refinement monoid, which is far more difficult.

The proof of Theorem~\ref{T:C*CXAKRR0} is very similar to the one of Theorem~\ref{T:CXAKRR0}, now getting the $P$-indexed diagram~$\vA$ from Proposition~\ref{P:cDliftedC*} instead of Proposition~\ref{P:cDliftedExch}. The right larder part requires no change. The left larder part needs to be modified in the obvious way: the category~$\xA_K$ needs to be replaced by the closure~$\xA$ of~$\cA$ under finite products and C*-direct limits. As the functor~$\rV$ preserves finite direct products and C*-direct limits, $(\rV(X),[1_X])$ is a refinement monoid with order-unit of index at most two, for each $X\in\xA$. As in the proof of Lemma~\ref{L:cAKLeftLard}, the
quadruple $(\xA,\CM,\CM^\Rightarrow,\Phi)$ is a left larder. As every member of~$\cA$ has real rank either~$0$ or~$1$ while it has stable rank~$1$ (cf. Proposition~\ref{P:RRSRAC}), this is also the case for every direct limit of finite products of members of~$\cA$, in particular for the condensate $E:=\xF(X)\otimes\vA$. The remaining changes that need to be applied to the proof of Theorem~\ref{T:CXAKRR0} are trivial.

\section{Open problems}\label{S:Pbs}

\begin{problem}\label{Pb:C*0vsReg}
Does any of the classes $\setm{M}{(\exists R\text{ regular ring)}(M\cong\rV(R))}$ and\newline
$\setm{M}{(\exists A\text{ C*-algebra of real rank~$0$})(M\cong\rV(A))}$ contain the other? And on positive cones of dimension groups?
\end{problem}

The ``combinatorial'' side of the second part of Problem~\ref{Pb:C*0vsReg} reads as follows.

\begin{problem}\label{Pb:C*0vsRegDiagr}
Are there finite, lattice-indexed commutative diagrams of simplicial monoids that can be lifted, with respect to the functor~$\rV$, by diagrams in one of the classes~$\Reg$ and~$\xRR$ but not the other?
\end{problem}

The results and methods of Elliott~\cite{Elli76} and Goodearl and Handelman~\cite{GoHa86} imply that the nonstable K-theories of regular rings and of C*-algebras of real rank~$0$ (and also of AF C*-algebras) agree on dimension groups of cardinality at most~$\aleph_1$. Our next question is about the remaining cardinality gap~$\aleph_2$.

\begin{problem}\label{Pb:Al2}
Can the bound $\aleph_3$ be improved to~$\aleph_2$ in Theorems~\ref{T:CXAKRR0} and~\ref{T:C*CXAKRR0}?
\end{problem}

As to the non-representability by regular rings, it looks plausible that a positive solution to Problem~\ref{Pb:Al2} would follow from the methods of Wehrung~\cite{NonMeas}. In that paper, the fact that the principal right ideals in a regular ring form a lattice is crucial. As the analogous result for C*-algebras of real rank~$0$ does not hold, it sounds unlikely that the methods of~\cite{NonMeas} are ready to help finding a solution to Problem~\ref{Pb:Al2} for those algebras. On the other hand, it is plausible that one may be able to use special functors, similar to those involved in Gillibert and Wehrung \cite[Chapter~5]{GiWe2} (which is pure lattice theory!), for which a suitable analogue of Theorem~\ref{T:NoLift} would remain valid.

Every conical \cm\ with order-unit is isomorphic to $\rVu(R)$ for some unital hereditary ring~$R$: this is proved in Theorems~6.2 and~6.4 of Bergman~\cite{Berg74} for the finitely generated case, and in Bergman and Dicks \cite[page~315]{BeDi78} for the general case. The general, non-unital case is proved in Ara and Goodearl \cite[Proposition~4.4]{ArGo11}. In light of this result, the following problem is natural.

\begin{problem}\label{Pb:RingsFunct}
Does there exist a functor~$\lift$, from the category of conical \cm s with monoid \emph{homomorphisms} to the category of rings and ring homomorphisms, such that $\rV\circ\lift\cong\id$?
\end{problem}

We do not even know the answer to the following restricted version of Problem~\ref{Pb:RingsFunct}.

\begin{problem}\label{Pb:C*AlgVLift}
Does there exist a functor~$\lift$, from the category of simplicial monoids with monoid \emph{homomorphisms} to the category of C*-algebras and \Chom s, such that $\rV\circ\lift\cong\id$?
\end{problem}

The lifting results of~$\vLu$ obtained in Theorems~\ref{T:LuLift} and~\ref{T:LuLiftC*} also suggest the following problems.

\begin{problem}\label{Pb:ExchRingsLift}
Does there exist a functor~$\lift$, from the category of simplicial monoids with normalized monoid \emph{embeddings} to the category of exchange rings and ring homomorphisms, such that $\rV\circ\lift\cong\id$?
\end{problem}

\begin{problem}\label{Pb:C*AlgSimpLift}
Does there exist a functor~$\lift$, from the category of simplicial monoids with normalized monoid \emph{embeddings} to the category of C*-algebras and \Chom s, such that $\rV\circ\lift\cong\id$?
\end{problem}

Of course, all the problems above have unital versions, which are open as well.

\end{document}